\documentclass[11pt,final]{iopart}
\usepackage{amssymb}
\usepackage{amsfonts}
\usepackage{amsthm}
\usepackage{iopams}
\usepackage{graphicx,subfigure}
\usepackage[dvips]{epsfig}
\usepackage{float}
\usepackage[bottom]{footmisc}

\usepackage{perpage}
\MakePerPage{footnote}
\usepackage[title,titletoc,toc]{appendix}
\usepackage{prettyref}
\usepackage{appendix}
\usepackage{wrapfig}
\usepackage{showkeys}
\usepackage{todo}
\usepackage{prettyref}
\usepackage{verbatim}
\usepackage[latin1]{inputenc}
\usepackage{subfigure}
\usepackage[dvips]{epsfig}
\usepackage{algorithm}
\usepackage{algpseudocode}
\usepackage{verbatim}
\newrefformat{eq}{(\ref{#1})}
\newrefformat{fig}{Figure~\ref{#1}}
\usepackage[pdfpagelabels]{hyperref}
\hypersetup{colorlinks=true,
linkcolor=blue,
citecolor=blue,
plainpages=false}
\newtheorem{theorem}{\sc Theorem}[section]
\newtheorem{propn}[theorem]{\sc Proposition}
\newtheorem{ppty}[theorem]{\sc Property}

\newtheorem{assump}[theorem]{\bf Assumption}

\newtheorem{definition}[theorem]{\sc Definition}

\newcommand{\argmin}{{\rm arg\,min}}

\eqnobysec

\makeindex

\newcommand{\cJ}{{\mathcal J}}

\newcommand{\cX}{{\mathcal X}}
\newcommand{\cY}{{\mathcal Y}}

\newcommand{\tu}{\tilde{u}}

\newcommand{\hu}{\hat{u}}

\newcommand{\hw}{\hat{w}}

\newcommand{\hz}{\hat{z}}

\newcommand{\R}{\mathbb{R}}

\newcommand{\norm}[1]{|| #1||}

\newcommand{\ba}{\begin{array}}
\newcommand{\ea}{\end{array}}
\newcommand{\be}{\begin{equation}}
\newcommand{\ee}{\end{equation}}
\newcommand{\bea}{\begin{eqnarray}}
\newcommand{\eea}{\end{eqnarray}}
\newcommand{\beq}{\begin{equation}}
\newcommand{\eeq}{\end{equation}}
\newcommand{\bqt}{\begin{quote}}
\newcommand{\eqt}{\end{quote}}

\begin{document}

%
\title[Primal Dual Algorithms]
{New Pair of Primal Dual Algorithms for Bregman 
Iterated Variational Regularization}

%
\author{Erdem Altuntac}

\address{Universite Libre de Bruxelles,
D\`{e}partement de Math\`{e}matique, Boulevard du 
Triomphe, 1050, Bruxelles, Belgium}


\ead{\mailto{Erdem.Altuntac@ulb.ac.be}}

\begin{abstract}

Primal-dual splitting involving proximity operators in order to
be able to find some approximation to the minimizer 
for a general form of
Tikhonov type functional is in the focus of this work. 
This approximation is produced by a pair of 
iterative variational regularization procedures.

Under the assumption of some variational source condition (VSC), 
total error estimation both in the iterative sense
and in the continuous sense has been analysed separately.
Rates of convergence 
will be obtained in terms some concave and positive definite index function. 
Of the choice of the penalty term, we are interested in Bregman distance
penalization associated with the non-smooth
total variation (TV) functional. Furthermore, following up the lower and bounds
defined for the regularization parameter, some deterministic choice of the 
regularization parameter is given explicitly. 
It is in the emphasis of this work that the regularization
parameter obeys {\em Morozov`s discrepancy principle} (MDP) 
in order for the stability analysis of regularized solution.

In the computerized environment, the algorithms are verified as iterative
regularization methods by applying it to an atmospheric 
tomography problem named as {\em GPS-Tomography}.
Apart from this 3-D tomographic inverse
problem, we also apply the algorithms to some 2-D 
conventional tomographic image reconstruction 
problems in order to be able test algorithms` capability of capturing the 
details and observe that algorithms behave as iterative regularization
procedures.

\end{abstract}

\bigskip


\section{Introduction}

Following up a recent work in \textbf{\cite{Altuntac18}}, we introduce
the primal dual algorithms in the simpler forms without nested loops.
Both algorithms are to be used for the purpose of finding iteratively
regularized approximations for the inverse ill-posed problems. 
Stability analyses are reduced to be in the emphasis of the stability of 
the iteratively regularized approximations of the regularized 
minimizer for some general form of Tikhonov functional.

In general terms, regularization theory deals with approximation
of some ill-posed inverse problem by a family of parametrized 
well-posed problems. Traditional quadratic-Tikhonov regularization
\textbf{\cite{Tikhonov63, TikhonovArsenin77}} has been
well established and analyzed \textbf{\cite{Engl96}}.
This work focuses on the analysis of some iterative regularization
procedures.

Proximal mapping algorithms and Bregman iterated regularization procedures
have been commonly known. Here, Bregman distance plays the role of penalization
in the corresponding minimization problems. Having a look at the literature,
firstly in the work \textbf{\cite{OsherBurger05}}, 
Bregman iteration has been proposed for providing solution
to the basis pursuit problem. 
In a recent study by Sprung and Hohage {\em et al.}, 2017,
\textbf{\cite{SprungHohage17}}, authors have investigated
convergence rates results for the such objective functionals
with the Bregman distance as penalty term.
Having optimization algorithms in the field of inverse problems
as iterative regularization method has also become popular. Authors in 
\textbf{\cite{GarrigosVilla18}} have proposed some primal-dual
algorithm, wherein the convergence has been studied for the given
noiseless measurement data. We consider linear, inverse 
ill-posed problems in the general form. Stability
of the algorithms will be developed in the context
of convex variational regularization and be verified in the Hadamard sense. 
Main results of our work
are derived in the case of noisy measurement and are in the best 
interest of variational regularization theory.


\section{Notations and Mathematical Setting}
\label{notations}

Over the finite dimensional Hilbert spaces $\cX = \R^{N}$
and $\cY = \R^{M}$, let us be given some linear, injective,
forward operator $T : \cX \rightarrow \cY.$ In this work, 
we concentrate on the numerical solution for the linear inverse ill-posed 
problem of the form
\bea
\label{inverse_problem}
\delta\xi + T u = v^{\delta},
\eea
with some iterative regularization procedure.
Here, the given noisy data $v^{\delta} \in \cY$ and 
the noise model is denoted by $\xi$ with the noise
magnitude $\delta.$

Non-negativity constraint
on our targeted data $u$ is imposed.
Then, the constraint domain is treated as the indicator function
$h : \cX \longrightarrow \{0,1\}$ defined by
\beq
\label{indicator}
h(u) = 1_{\Omega}(u) := \left\{ \begin{array}{rcl}
0 & \mbox{, for} & u \in \Omega \subset \cX \\ 
\infty  & \mbox{, for} & u \notin \Omega \subset \cX .
\end{array}\right.
\eeq 

Throughout the work, unless otherwise stated,
the notation $\norm{\cdot}$ without any subscript
will be used as to denote the usual Euclidean norm.
Let $\sigma(T^{T}T)$ be the
spectrum of $T^{T}T$ is the set of those 
$\sigma_{k} \in \R^{N}.$ 
Then, for the finite dimensional forward operator 
$T : \R^{N} \rightarrow \R^{M}$ where $M < N,$
we define
\begin{displaymath}
\norm{T} := \max_{1\leq k \leq M} 
\left\{ \sqrt{\sigma_{k}} \right\}.
\end{displaymath}

Below, we give two norm estimations that will be in use
of our mathematical development. For some $u_1,u_2 \in \cX$ 
and $\lambda \in \R,$ the following equality holds, 
\textbf{\cite[Eq. (2.1)]{Takahashi13}},
\bea
\label{strct_cvx_eq}
\norm{\lambda u_1 + (1-\lambda)u_2}^2 = \lambda\norm{u_1}^2 + (1-\lambda)\norm{u_2}^2
-\lambda(1-\lambda)\norm{u_1 - u_2}^2.
\eea
Also, nonexpansiveness of the misfit term provides, 
\bea
\norm{T^{T}T(u_{1} - u_{2})}^2 & \leq & 
\norm{T}^2\norm{T(u_{1} - u_{2})}^2
\nonumber\\
& = &\norm{T}^2\langle T(u_{1} - u_{2}) , 
T(u_{1} - u_{2}) \rangle
\nonumber\\
& = & \norm{T}^2\langle T^{T}T(u_{1} - u_{2}) , 
u_{1} - u_{2} \rangle , 
\nonumber
\eea
which implies
\bea
\label{LipschitzEstMisfit}
-\langle u_{1} - u_{2} , 
T^{T}T(u_{1} - u_{2}) \rangle \leq 
-\frac{1}{\norm{T}^2}\norm{T^{T}T(u_{1} - u_{2})}^2 
\eea

For some function $f : \cX \rightarrow \cY$ and some point 
$x$ in the domain of $f,$ the {\em subdifferential} of $f$ at 
$x^{\prime},$ denotes $\partial f(x^{\prime})$ is defined by
\bea
\label{def_subdiff}
\partial f(x^{\prime}) := \left\{ \eta \in \cX^{\ast} :
f(x) - f(x^{\prime}) \geq \langle \eta , x - x^{\prime} \rangle
\mbox{ for all } x \in \cX \right\}.
\eea

\begin{definition}\textbf{[Generalized Bregman Distances]}
Let $\cJ : \cX \rightarrow \R_{+} \cup \{ \infty \}$ be a convex functional 
with the subgradient $q^{\ast} \in \partial \cJ(u^{\ast}).$ 
Then, for $u, u^{\ast} \in \cX,$ Bregman distance associated with the functional 
$J$ is defined by
\bea
\label{bregman_divergence_intro}
D_{\cJ} : & \cX \times \cX & \longrightarrow \R_{+}
\nonumber\\
& (u , u^{\ast}) & \longmapsto D_{\cJ}(u , u^{\ast}) 
:= \cJ(u) - \cJ(u^{\ast}) - 
\langle q^{\ast} , u - u^{\ast} \rangle .
\eea
It is well known that the Bregman distance does not satisfy
symmetry,
\begin{displaymath}
D_{\cJ}(u , u^{\ast}) \neq D_{\cJ}(u^{\ast} , u),
\end{displaymath}
and for the defined convex functional $\cJ$
\begin{displaymath}
D_{\cJ}(u , u^{\ast}) \geq 0.
\end{displaymath}
\end{definition}

With all these tools stated,
we consider the following objective functional,
\bea
\label{obj_functional2}
F_{\alpha}: & \cX & \times \cY \longrightarrow \R_{+}
\nonumber\\
& (u &, v^{\delta}) \longmapsto F_{\alpha}(u,v^{\delta}) := 
\frac{1}{2} \norm{Tu - v^{\delta}}^2 +  
\alpha D_{\cJ}(u , u_{0}) + h(u),
\eea
with some initial estimation $u_{0} \in \cX.$ In particular, we
associate the Bregman distance penalty term with
the total variation (TV) functional defined by
\beq
\label{TV_penalty}
TV(u, \Omega) = \cJ(u) := \int_{\Omega} \vert \nabla u(x) \vert dx 
\approxeq \sum_{i} \vert \nabla_{i} u \vert ,
\eeq
which is, in 3D case, $i = (i_x, i_y, i_z).$ 
For the sake of following the further calculations easily in the future
developments of this work, we introduce TV in the composite form
\bea
\label{composite_TV}
J(u) = g(D u) \mbox{ where, }
g(\cdot) = \norm{\cdot}_{1} \mbox{ with }
D(\cdot) = \nabla (\cdot).
\eea
Thus,
\bea
\label{subdiff_TV}
\partial J(u) = D^{\ast} \partial g(Du).
\eea

\subsection{Overview on the iterative regularization and the choice of the regularization parameter}
\label{it_regularization}

By an iterative procedure involving some iteration operator $R_{I},$
\cite[Ch. 6]{BenningBurger17},
we aim to construct some approximation to the given inverse
ill-posed problem (\ref{inverse_problem})
\bea
\label{iterative_proc}
u_{i} = R_{I}(v^{\delta} , w_{i-1} , \Gamma),
\eea
where $w^{i-1}$ is the collection of dual variables used 
during $i-1$ iteration steps, and $\Gamma$ is the 
auxiliary parameters such as step-size, relaxation parameter, 
regularization parameter.

In the iterative regularization procedures, discrepancy principles act as
the stopping rules for the corresponding algorithms, 
\textbf{\cite[Section 6]{BenningBurger17}}.

\begin{definition}{[Morozov`s Discrepancy Principle (MDP), 
\cite[Def. 6.1]{BenningBurger17}]}

Given deterministic noise model 
$\norm{v^{\dagger}-v^{\delta}} \leq \delta,$ if we choose
$\tau > 1$ and $i^{\ast} = i^{\ast}(\delta , v^{\delta})$ such that
\bea
\norm{T u_{i^{\ast}} - v^{\delta}} \leq \tau \delta < 
\norm{T u_{i} - v^{\delta}}
\eea
is satisfied for $u_{i^{\ast}} = 
R_{I}(v^{\delta} , w_{i^{\ast}-1} , \vec{\alpha})$ and 
$u_{i} = R_{I}(v^{\delta} , w_{i-1} , \vec{\alpha})$ for all
$i < i^{\ast},$ then $u_{i^{\ast}}$ is said to satisfy {\em Morozov`s 
discrepancy principle}.
\end{definition}
Following up MDP, some immediate consequences 
can be given below,
\bea
\label{consequence_MDP}
\norm{T u_{i^{\ast}}^{\delta} - T u^{\dagger}} \leq (\tau + 1)\delta ,
\eea
likewise,
\bea
\tau\delta \leq \norm{T u_{i} - v^{\delta}} & \Rightarrow & 
\tau\delta \leq \norm{T u_{i} - T u^{\dagger}} + \delta
\nonumber\\ 
& \Rightarrow & 
(\tau - 1)\delta \leq \norm{T u_{i} - T u^{\dagger}}
\nonumber\\ 
& \Rightarrow & 
(\tau - 1)^2\delta^2 \leq \norm{T u_{i} - T u^{\dagger}}^2
\nonumber\\ 
& \Rightarrow &
-\norm{T u_{i} - T u^{\dagger}}^2 \leq -(\tau - 1)^2\delta^2
\label{consequence_MDP2}
\eea

Our primal-dual splitting algorithms involve proximal mapping that
is defined below.
\begin{definition}\textbf{[Proximal mapping]}
\label{def_prox}
Let $\cJ : \R^{N} \rightarrow \R \cup \{ \infty \}$
be a proper, convex, lower-semicontinious function.
Then $\mathrm{prox}_{\cJ}$ is defined as the unique minimizer
\bea
\nonumber
\mathrm{prox}_{\cJ}(\tu) := 
\argmin_{u \in \R^{N}} \cJ(u) + \frac{1}{2} \norm{u - \tu}^2.
\eea
\end{definition}

Measuring the deviation of the regularized solution 
$u_{\alpha}^{\delta}$ 
from the minimum norm solution $u^{\dagger}$ 
by the {\em a priori} and {\em a posteriori}
strategies for the choice of the regularization parameter
in Banach spaces with the VSC has been widely studied,

The objective of convergence and convergence rates results in the 
regularization theory is to be able find some stable bound for the total 
error estimation function defined by
\bea
E: & \cX \times \cX & \longrightarrow \R_{+}
\nonumber\\
\label{total_err_est}
& (u_{\alpha}^{\delta} , u^{\dagger}) & \longrightarrow
E(u_{\alpha}^{\delta} , u^{\dagger}) := 
\Lambda\norm{u_{\alpha}^{\delta} - u^{\dagger}},
\eea 
where the coefficient $\Lambda \in \R_{+}$ depends 
on the functional properties of the data on the pre-image space.
Such estimation requires the knowledge of the smoothness of the 
minimum norm solution $u^{\dagger}$ which is some source condition
in the form of varitional inequality, 
\textbf{\cite{BenningSchoenlib17},
\cite{Flemming18}, \cite{Grasmair10},   
\cite[Eq. (1.4)]{HofmannMathe12}, \cite{HofmannKaltenbacher07},
\cite[Section 4]{Kindermann16}
\cite[Theorem 2.60 - (g), Subsection 3.2.4]{Schuster12}}.
Convergence and convergence rates results, 
or in other words the total error estimation, are derived in terms of 
a concave, monotonically increasing, positive definite index function
that is a part of the VSC expression. Following up the work 
\textbf{\cite{Altuntac18}}, the form of VSC that will be used in
our analysis is given below. Also in the aforementioned works that 
have been dedicated to variational regularization, derivation of 
noise dependent error estimation following from VSC has been well 
explained. In our work, any data that is in the constraint domain
is assumed to satisfy VSC.
\begin{assump}
\label{assump_conventional_variational_ineq}
\textbf{[Variational Source Condition]}
Let $T : \cX \rightarrow \cY$ be linear, injective 
forward operator and $v^{\dagger} \in \mathrm{range}(T).$ 
There exists some constant $\sigma \in (0 , 1]$
and a concave, monotonically increasing
index function $\Psi$ with $\Psi(0) = 0$ and 
$\Psi : [0 , \infty) \rightarrow [0 , \infty)$
such that for $q^{\dagger} \in \partial \cJ(u^{\dagger})$
the minimum norm solution $u^{\dagger}\in\mathrm{BV}(\Omega)$
satisfies
\bea
\label{variational_ineq}
\sigma\norm{u-u^{\dagger}} \leq 
\cJ(u) - \cJ(u^{\dagger})
+ \Psi\left( \norm{T u - T u^{\dagger}} \right)  
\mbox{, for all }
u \in \cX .
\eea 
\end{assump}


\section{Subdifferential Characterization}

Both algorithms will evolve from the subdifferential characterization
of the regularized minimizer of the objective functional 
(\ref{obj_functional2}). Then, 
by its first order optimality condition,
\bea
\nonumber
0 \in \partial 
F_{\alpha_{\ast}}(u_{\alpha_{\ast}}^{\delta} , v^{\delta}),
\eea
which implies,
\bea
\label{char_opt_cond}
0 = T^{T}
(T u_{\alpha_{\ast}}^{\delta} - v^{\delta})
+ \alpha_{\ast}\partial \cJ(u_{\alpha_{\ast}}^{\delta}) - 
\alpha_{\ast}\partial \cJ(u_{0}) + \hz 
\mbox{, where }
\hz \in \partial h(u_{\alpha_{\ast}}^{\delta}).
\eea
Furthermore, recall the settings in (\ref{composite_TV}) 
and (\ref{subdiff_TV}) to represent (\ref{char_opt_cond})
in the following form
\bea
\nonumber
0 = T^{T}(T u_{\alpha_{\ast}}^{\delta} - v^{\delta})
+ \alpha_{\ast}D^{T} \hw_{\alpha_{\ast}}^{\delta} 
- \alpha_{\ast}D^{T} \hw_{0} + \hz ,
\eea
where $\hw_{\alpha}^{\delta} \in \partial g(Du_{\alpha}^{\delta})$
and likewise
$\hw_{0} \in \partial g(Du_{0}).$

\begin{theorem}\cite[Theorem 4.1]{Altuntac18}
[Subgradient characterization of the regularized solution]
\label{thrm_subgrad_char}

For any positive valued $\alpha, \nu, \mu,$
the regularized minimizer $u_{\alpha}^{\delta}$ of the 
objective functional (\ref{obj_functional2}) is characterized by
\bea
\label{eq_subgrad_char}
\left\{ \begin{array}{rcl}
u_{\alpha}^{\delta} &=& \mathrm{prox}_{\mu h}
\left[u_{\alpha}^{\delta}  
- \mu \left( T^{T}(Tu_{\alpha}^{\delta} - v^{\delta})
+ \alpha_{\ast}D^{T}(\hw_{\alpha}^{\delta} - \hw_{0} ) \right) 
\right]
\\
\hw_{\alpha}^{\delta} &=& \mathrm{prox}_{\nu g^{\ast}} 
\left( \hw_{\alpha}^{\delta} + 
\nu Du_{\alpha}^{\delta} \right),
\end{array}\right.
\eea
with $\hw_{0} = \partial \norm{D u_{0}}_{1}.$
\end{theorem}


\section{The Primal-Dual Algorithms}

We introduce a pair of algorithms involving proximal mappings.
Both algorithms aim to provide approximation for the regularized
minimizer of the objective functional (\ref{obj_functional2}).
Algorithm \ref{algorithmNPA-I} can be interpreted as a direct 
discrete form of the subdifferential characterization given above.
However, Algorithm \ref{algorithmNPA-II} is endowed with
some projected convex extrapolation on a line segment due to 
the choice of the relaxation parameter $\lambda$.

Further than investigating whether the regularized iterations are better 
approximations, we are also interested in the convergences of those 
approximations towards the minimum norm solution.


\begin{algorithm}
  \caption{Primal Dual Algorithm}
  \label{algorithmNPA-I}
  \begin{algorithmic}[1]
    \Procedure{Define $\underline{\tau}, \overline{\tau}, \alpha_0$}
    {}
      \State \textbf{initiation} Given $u_{0},$
      calculate $w_{0} \in \partial \norm{D u_{0}}_1$ and set
        $w_{1} = w_{0}$
      \While{$\underline{\tau}\delta\leq
      \norm{T u_{i^{\ast}} - v^{\delta}}_{\cY}\leq
      \overline{\tau}\delta$ or 
      $\norm{u_{i^{\ast}} - u^{\dagger}}_{\cX}/
      \norm{u^{\dagger}}_{\cX} \leq \epsilon$}
        \State 
        \label{parametersII}
        $\mu_i, \nu_{i}, \alpha_i$ \Comment{Parameter update}
         \State 
         \label{algorithm_primalstep}
         $u_{i+1} = \mathrm{prox}_{\mu h}                                                                                                                                   
        \left[u_{i} - \mu_i\left(T^{T}
        (T u_{i} - v^{\delta}) + 
        \alpha_{i}D^{T}(w_{i} - w_{0})\right) \right]$ 
        \Comment{Primal update}
        \State
        \label{algorithm_dualstep}
         $w_{i+1} = 
        \mathrm{prox}_{\nu_i g^{\ast}} 
        \left( w_{i} + \nu_i D u_{i+1} \right)$ \Comment{Dual update}
        \State\textbf{calculate} $D^{T}(w_{i} - w_{0})$
      \EndWhile     
    \EndProcedure
  \end{algorithmic}
\end{algorithm}

\begin{algorithm}
  \caption{Primal Dual Algorithm With Convex Extrapolation Over Some Line}
  \label{algorithmNPA-II}
  \begin{algorithmic}[1]
    \Procedure{Define $\underline{\tau}, \overline{\tau}, \alpha_0$}
    {}
      \State \textbf{initiation} Given $u_{0},$
      calculate $w_{0} \in \partial \norm{D u_{0}}_1$ and set
        $w_{1} = w_{0}$
      \While{$\underline{\tau}\delta\leq
      \norm{T u_{i^{\ast}} - v^{\delta}}_{\cY}\leq
      \overline{\tau}\delta$ or 
      $\norm{u_{i^{\ast}} - u^{\dagger}}_{\cX}/
      \norm{u^{\dagger}}_{\cX} \leq \epsilon$}
        \State 
        \label{parametersII}
        $\mu_i, \nu_{i}, \alpha_i$ \Comment{Parameter update}
         \State 
         \label{algorithm_primalstep}
         $\hu_{i+1} = \mathrm{prox}_{\mu h}                                                                                                                                   
        \left[u_{i} - \mu_i\left(T^{T}
        (T u_{i} - v^{\delta}) + 
        \alpha_{i}D^{T}(w_{i} - w_{0})\right) \right]$ 
        \Comment{Primal update}
        \State
        \label{algorithm_dualstep}
         $w_{i+1} = 
        \mathrm{prox}_{\nu_i g^{\ast}} 
        \left( w_{i} + \nu_i D \hu_{i+1} \right)$ \Comment{Dual update}
        \State\textbf{calculate} $D^{T}(w_{i} - w_{0})$

        \State \textbf{update} 
        \label{projected_linesearch}
        $u_{i + 1} = (1-\lambda)u_{i} + \lambda\hu_{i+1}$ 
        \Comment{Convex extrapolation with $\lambda \in (1,2)$}
      \EndWhile     

    \EndProcedure
  \end{algorithmic}
\end{algorithm}

\section{Stability Analysis of the Algorithms}
\label{iterative_convergence}

Iterative total error estimation can be decomposed in the following form
\bea
\norm{u_{i} - u^{\dagger}} \leq \norm{u_{i} - u_{\alpha}^{\delta}} 
+ \norm{u_{\alpha}^{\delta} - u^{\dagger}}
\nonumber
\eea
The term on the far right has been very well analysed in the context of
variational regularization. In what follows, we will focus on
the term on the left hand side and the first term on the right hand side.
Before stating that the $u_{i}$ is the approximation to the regularized 
minimizer $u_{\alpha}^{\delta}$ of the objective functional, necessary 
and sufficient conditions for the boundedness of 
$\norm{u_{i} - u_{\alpha}^{\delta}}$ must be established.
To this end, we shall give a pair of
observations on the iterative approximations produced by 
the both algorithms. The assertions in the following formulations 
describe how the parameters in the algorithms must be chosen. 
Further assumption for the theoretical developments is the initial
guess. In order to overcome the mathematical difficulties, 
it is always assumed that the initial guess of the objective 
functional and the initial guess for the algorithms are the same.

\subsection{Iterative approximations of the regularized minimizer}

The primary tool to study stability of the
iteratively regularized approximations that are produced by 
the proximal gradient algorithms is formulated below.
\begin{ppty}\cite[Lemma 1]{LorisVerhoeven11}
\label{property_prox_update}
If $x^{+} = \mathrm{prox}_{g}(x^{-} + \Delta)$, then
for any $y \in \R^{N},$
\bea
\label{prox_update_assertion}
\norm{x^{+} - y}^{2} \leq \norm{x^{-} - y}^{2} - \norm{x^{+} - x^{-}}^2
+ 2 \langle x^{+} - y , \Delta \rangle + 2 g(y) - 2 g(x^{+}).
\eea
\end{ppty}
Before stability analysis, we give some observations on the primal
variables produced by the both algorithms.
\begin{propn}
\label{propn_better_approxI}
Let the step-length $\mu_i$ satisfy 
$\mu_i \leq \frac{2}{\norm{T}^2}.$ 
Furthermore, let the iterative regularization 
parameter be chosen as $\alpha_{i} = i(\delta , v^{\delta})$
and $\nu_{i} \leq \frac{i(\delta , v^{\delta})^2}{\mu_{i}^2}.$
Then, iteratively regularized primal variable
$u_{i + 1}$ that is produced by Algorithm \ref{algorithmNPA-I}
is a better approximation of 
$u_{\alpha}^{\delta}$ than
$u_{i}$ for each $i = 0, 1, 2, \cdots$.
\end{propn}

\begin{proof}
Following first two error estimations between
the final update $u_{i + 1}$ of the Algorithm 
\ref{algorithmNPA-I} and the regularized minimizer
$u_{\alpha}^{\delta}$ of the objective functional (\ref{obj_functional2}) according to the Property \ref{property_prox_update} is given by
\bea
\norm{u_{i + 1} - u_{\alpha}^{\delta}}^2 \leq 
\norm{u_{i} - u_{\alpha}^{\delta} }^{2} - 
\norm{ u_{i + 1} - u_{i} }^{2}
& - & 2\mu_{i} \langle u_{i + 1} - u_{\alpha}^{\delta} , 
T^{T}(T u_{i} - v^{\delta}) \rangle .
\nonumber\\
& - & 2\mu_{i}\alpha_{i}\langle u_{i + 1} - u_{\alpha}^{\delta} , 
D^{T}(w_{i} - w_{0}) \rangle .
\eea
Likewise, still by Property \ref{property_prox_update},
\bea
\norm{u_{\alpha}^{\delta} - u_{i + 1}}^2 \leq 
\norm{u_{\alpha}^{\delta} - u_{i + 1}}^2 - 
\norm{u_{\alpha}^{\delta} - u_{\alpha}^{\delta}}
& - & 2\mu \langle u_{\alpha}^{\delta} - u_{i + 1} , 
T^{T}(T u_{\alpha}^{\delta} - v^{\delta}) \rangle 
\nonumber\\
& - & 2\mu\alpha \langle u_{\alpha}^{\delta} - u_{i + 1} , 
D^{T}(w_{\alpha} - w_{0})\rangle .
\eea
After the necessary simplifications, these both estimations in total return,
\bea
\norm{u_{i + 1} - u_{\alpha}^{\delta}}^2 \leq 
\norm{u_{i} - u_{\alpha}^{\delta} }^{2} - 
\norm{ u_{i + 1} - u_{i} }^{2}  
& + & 2\mu_{i}\langle u_{i+1} - u_{\alpha}^{\delta} , 
T^{T}T(u_{\alpha}^{\delta} - u_{i}) \rangle
\nonumber\\
& - & 2\mu_{i}\alpha\langle u_{\alpha}^{\delta} - u_{i+1} , 
D^{T}(w_{\alpha} - w_{0}) \rangle
\nonumber\\
& - & 2\mu_{i}\alpha_{i}\langle u_{i + 1} - 
u_{\alpha}^{\delta} , D^{T}(w_{i} - w_{0}) \rangle .
\label{primalvar_AlgI_total0}
\eea
Some useful upper bound for
the inner product on the right hand side of the 1st line is 
given below,
\bea
\label{primal_innerprod0}
2\mu_{i}\langle u_{i+1} - u_{\alpha}^{\delta} , 
T^{T}T(u_{\alpha}^{\delta} - u_{i}) \rangle & = &
2\mu_{i}\langle u_{i+1} - u_{i} , 
T^{T}T(u_{\alpha}^{\delta} - u_{i})\rangle
- 2\mu_{i}\langle u_{\alpha}^{\delta} - u_{i}, 
T^{T}T(u_{\alpha}^{\delta} - u_{i})\rangle
\nonumber\\
& \leq &
\norm{u_{i+1} - u_{i}}^2 + 
\mu_{i}^2\norm{T^{T}T(u_{\alpha}^{\delta} - u_{i})}^2
- 2 \mu_{i}\frac{1}{\norm{T}^2}
\norm{T^{T}T(u_{\alpha}^{\delta} - u_{i})}^2
\nonumber\\
& = & \norm{u_{i+1} - u_{i}}^2 + 
\mu_{i}\left(\mu_{i} - \frac{2}{\norm{T}^2}\right)
\norm{T^{T}T(u_{\alpha}^{\delta} - u_{i})}^2 ,
\eea
where we have used (\ref{LipschitzEstMisfit}). Also, total of
the inner products on the 2nd and the 3rd lines of 
(\ref{primalvar_AlgI_total0}) can be given
in a simpler form. Thus, 
\bea
\norm{u_{i + 1} - u_{\alpha}^{\delta}}^2 \leq 
\norm{u_{i} - u_{\alpha}^{\delta} }^{2} & + &
\mu_{i}\left(\mu_{i} - \frac{2}{\norm{T}^2}\right)
\norm{T^{T}T(u_{\alpha}^{\delta} - u_{i})}^2 
\nonumber\\
& - & 2\mu_{i}(\alpha - \alpha_{i})
\langle u_{\alpha}^{\delta} - u_{i + 1} , D^{T}w_{\alpha}^{\delta}\rangle 
- 2\mu_{i}\alpha_{i}\langle u_{\alpha}^{\delta} - u_{i + 1} , 
D^{T}(w_{\alpha}^{\delta} - w_{i})\rangle
\nonumber\\
\label{primal_final_estI}
\eea
Analagous estimations on the dual iterative variable $w_{i+1}$ and
the dual variable $w_{\alpha}^{\delta}$ can be derived as follows,
\bea
\norm{w_{i+1} - w_{\alpha}^{\delta}}^2 \leq 
\norm{w_{i} - w_{\alpha}^{\delta}}^2 
- \norm{w_{i+1} - w_{i}}^2
+ 2\nu_{i}\langle w_{i+1} - w_{\alpha}^{\delta} , 
D(u_{i+1} - u_{\alpha}^{\delta}) \rangle .
\label{dual_varI_total0}
\eea
We rewrite the inner product on the right hand side 
\bea
2\nu_{i}\langle w_{i+1} - w_{\alpha}^{\delta} , 
D(u_{i+1} - u_{\alpha}^{\delta}) \rangle = 
2\nu_{i}\langle w_{i+1} - w_{i} , 
D(u_{i+1} - u_{\alpha}^{\delta}) \rangle + 
2\nu_{i}\langle w_{i} - w_{\alpha}^{\delta} , 
D(u_{i+1} - u_{\alpha}^{\delta}) \rangle
\nonumber\\
\label{dual_varI_innerprod}
\eea
Now, $\nu_{i}$ times (\ref{primal_final_estI}) plus
$\mu_{i}\alpha_{i}$ times (\ref{dual_varI_total0}) 
with taking into account (\ref{dual_varI_innerprod}) 
will result in further term reduction,
\bea
\nu_{i}\norm{u_{i + 1} - u_{\alpha}^{\delta}}^2 & \leq & 
\nu_{i}\norm{u_{i} - u_{\alpha}^{\delta} }^{2} + 
\nu_{i}\mu_{i}\left(\mu_{i} - \frac{2}{\norm{T}^2} \right)
\norm{T^{T}T(u_{\alpha}^{\delta} - u_{i})}^2 + 
\mu_{i}\alpha_{i}\norm{w_{i} - w_{\alpha}^{\delta}}^2 
\nonumber\\
& - & \mu_{i}\alpha_{i}\norm{w_{i+1} - w_{i}}^2
+ 2\nu_{i}\mu_{i}(\alpha - \alpha_{i})
\langle u_{\alpha}^{\delta} - u_{i + 1} , D^{T}w_{\alpha}^{\delta}\rangle 
\nonumber\\
&+& 2\nu_{i}\mu_{i}\alpha_{i}\langle w_{i+1} - w_{i} , 
D(u_{i+1} - u_{\alpha}^{\delta}) \rangle ,
\label{primal_est_total0_AlgI}
\eea
where we have also dropped the term 
$\mu_{i}\alpha_{i}\norm{w_{i+1} - w_{\alpha}^{\delta}}^2$ 
from the left hand side
since the boundedness of the error estimation for the primal
variables are in the interest of this result. 
Also, the inner products, after using
$\langle \cdot , D^{T}(\cdot) \rangle = 
\langle D(\cdot) , \cdot \rangle$ 
are bounded in the following ways,
\bea
2\nu_{i}\mu_{i}\alpha_{i}\langle w_{i+1} - w_{i} , 
D(u_{i+1} - u_{\alpha}^{\delta}) \rangle 
& \leq &
(\mu_{i}\alpha_{i})^2\sqrt{\nu_{i}}\norm{w_{i+1} - w_{i}}^2 + 
\sqrt{\nu_{i}}\norm{D(u_{i+1} - u_{\alpha}^{\delta})}^2,
\nonumber
\\
2\nu_{i}\mu_{i}(\alpha - \alpha_{i})
\langle D(u_{\alpha}^{\delta} - u_{i + 1}) , 
w_{\alpha}^{\delta}\rangle  & \leq &
\left(\mu_{i}(\alpha - \alpha_{i})\right)^2
\sqrt{\nu_{i}}\norm{w_{i} - w_{\alpha}^{\delta}}^2 +
\sqrt{\nu_{i}}\norm{D(u_{i+1} - u_{\alpha}^{\delta})}^2.
\nonumber
\eea
In the light of these bounds,
after multiplying both sides by $\frac{1}{\nu_{i}}$
of (\ref{primal_est_total0_AlgI}),
\bea
\norm{u_{i + 1} - u_{\alpha}^{\delta}}^2 & \leq & 
\norm{u_{i} - u_{\alpha}^{\delta} }^{2} +
\mu_{i}\left(\mu_{i} - \frac{2}{\norm{T}^2} \right)
\norm{T^{T}T(u_{\alpha}^{\delta} - u_{i})}^2 + 
\frac{1}{\nu_{i}}\mu_{i}\alpha_{i}
\norm{w_{i} - w_{\alpha}^{\delta}}^2 
\nonumber\\
& + & \frac{1}{\sqrt{\nu_{i}}}\mu_{i}\alpha_{i}
(\mu_{i}\alpha_{i} - \frac{1}{\sqrt{\nu_{i}}})
\norm{w_{i+1} - w_{i}}^2 +\frac{1}{\sqrt{\nu_{i}}}
\left(\mu_{i}(\alpha - \alpha_{i})\right)^2
\norm{w_{i} - w_{\alpha}^{\delta}}^2
\nonumber\\
& + & \frac{2}{\sqrt{\nu_{i}}}
\norm{D(u_{i+1} - u_{\alpha}^{\delta})}^2.
\eea
Hence, asserted parameter choices will yield the result.
\end{proof}
As for the Algorithm \ref{algorithmNPA-II}, the similar 
observation is also formulated below.
\begin{propn}
\label{propn_better_approxII}
Let the relaxation parameter of the 
step \ref{projected_linesearch} of the 
Algorithm \ref{algorithmNPA-II} be 
$\lambda \in (1,2)$ and the step-length be 
$\mu_i \leq \frac{2}{\norm{T}^2}.$
Furthermore, let the iterative regularization 
parameter be chosen as $\alpha_{i} = i(\delta , v^{\delta})$
and $\nu_{i} \leq \frac{i(\delta , v^{\delta})^2}{\mu_{i}^2}.$
Then, the iteratively regularized primal variable
$u_{i + 1}$ is a better approximation of 
$u_{\alpha^{\ast}}^{\delta}$ than
$u_{i}$ for each $i = 0, 1, 2, \cdots$.
\end{propn}

\begin{proof}
We begin with applying the equality given in (\ref{strct_cvx_eq}) 
to the step \ref{projected_linesearch} of the algorithm,
\bea
\label{cvx_primalest}
\norm{u_{i + 1} - u_{\alpha}^{\delta}}^2 = 
(1-\lambda)\norm{u_{i} - u_{\alpha}^{\delta}}^2
+ \lambda\norm{\hu_{i+1} - u_{\alpha}^{\delta}}^2
-\lambda(1-\lambda)\norm{u_{i} - \hu_{i+1}}^2
\eea
Now, by Property \ref{property_prox_update} 
the error estimation between the primal variable
$\hu_{i+1}$ and the regularized minimizer $u_{\alpha}^{\delta}$
is given,
\bea
\norm{\hu_{i+1} - u_{\alpha}^{\delta}}^2 \leq 
\norm{u_i - u_{\alpha}^{\delta}}^2
- \norm{\hu_{i+1} - u_i}^2 
& - & 2\mu_i \langle \hu_{i+1} - u_{\alpha}^{\delta} , 
T^{T}(T u_{i} - v^{\delta}) \rangle 
\nonumber\\
& - & 2\mu_{i} \alpha_{i}\langle \hu_{i+1} - 
u_{\alpha}^{\delta} , 
D^{T}(w_{i} - w_{0}) \rangle .
\label{primal_est_it}
\eea
Still by Property \ref{property_prox_update},
analagous error estimation is given below
\bea
\norm{u_{\alpha}^{\delta} - \hu_{i+1}}^2 \leq 
\norm{u_{\alpha}^{\delta} - \hu_{i+1}}^2 
- \norm{u_{\alpha}^{\delta} - u_{\alpha}^{\delta}}^2
& - & 2\mu \langle u_{\alpha}^{\delta} - \hu_{i+1} , 
T^{T}(Tu_{\alpha}^{\delta} - v^{\delta} ) \rangle
\nonumber\\
& - & 2\mu\alpha
\langle u_{\alpha}^{\delta} - \hu_{i+1} , 
D^{T}(w_{\alpha}^{\delta} - w_{0}) \rangle ,
\nonumber
\eea
which is in other words, after multiplying both sides by 
$\frac{\mu_i}{\mu},$
\bea
0 \leq -2 \mu_{i} \langle u_{\alpha}^{\delta} - \hu_{i+1} , 
T^{T}(Tu_{\alpha}^{\delta} - v^{\delta} ) \rangle
- 2 \mu_{i}\alpha 
\langle u_{\alpha}^{\delta} - \hu_{i+1} , 
D^{T}(w_{\alpha}^{\delta} - w_{0}) \rangle .
\label{primal_est_cont}
\eea
Recall that we consider the constant valued initial guess
$u_{0}$ which makes the dual variable $w_{0}$ zero valued
by its definition. Keeping this in mind and
summing up (\ref{primal_est_it}) and (\ref{primal_est_cont})
will return
\bea
\norm{\hu_{i + 1} - u_{\alpha}^{\delta}}^2 \leq
\norm{u_{i} - u_{\alpha}^{\delta}}^2 
- \norm{ \hu_{i + 1} - u_{i} }^{2}
& + & 2\mu_{i}\langle \hu_{i+1} - u_{\alpha}^{\delta} , 
T^{T}T (u_{\alpha}^{\delta} - u_{i}) \rangle 
\nonumber\\
& + & 2\mu_{i} (\alpha - \alpha_{i})
\langle \hu_{i+1} - u_{\alpha}^{\delta} , 
D^{T}w_{\alpha}^{\delta} \rangle
\nonumber\\
& + &2\mu_{i}\alpha_{i}
\langle \hu_{i+1} - u_{\alpha}^{\delta} , 
D^{T}(w_{\alpha}^{\delta} - w_{i}) \rangle
\nonumber\\
\label{primal_var_total0}
\eea
Here, the first inner product on the first line of the 
right hand side can be bounded similar to the proof
above. So, one can quickly write down that
\bea
2\mu_{i}\langle u_{i+1} - u_{\alpha}^{\delta} , 
T^{T}T(u_{\alpha}^{\delta} - u_{i}) \rangle \leq
\norm{\hu_{i+1} - u_{i}}^2 + 
\mu_{i}\left(\mu_{i} - \frac{2}{\norm{T}^2}\right)
\norm{T^{T}T(u_{\alpha}^{\delta} - u_{i})}^2 
\nonumber
\eea
Then, (\ref{primal_var_total0}) is reduced to
\bea
\norm{\hu_{i + 1} - u_{\alpha}^{\delta}}^2 \leq
\norm{u_{i} - u_{\alpha}^{\delta}}^2 
+ \mu_{i}\left(\mu_{i} - \frac{2}{\norm{T}^2}\right)
\norm{T^{T}T(u_{\alpha}^{\delta} - u_{i})}^2
& + & 2\mu_{i} (\alpha - \alpha_{i})
\langle \hu_{i+1} - u_{\alpha}^{\delta} , 
D^{T}w_{\alpha}^{\delta} \rangle
\nonumber\\
& + &2\mu_{i}\alpha_{i}
\langle \hu_{i+1} - u_{\alpha}^{\delta} , 
D^{T}(w_{\alpha}^{\delta} - w_{i}) \rangle
\nonumber
\eea
This estimation together with (\ref{cvx_primalest}), 
after putting the term 
$\norm{\hu_{i+1} - u_{\alpha}^{\delta}}^2$ on the other 
right hand side, in total
\bea
\norm{u_{i + 1} - u_{\alpha}^{\delta}}^2 \leq 
(2 - \lambda)\norm{u_{i} - u_{\alpha}^{\delta}}^2 + 
(\lambda - 1)\norm{\hu_{i+1} - u_{\alpha}^{\delta}}^2 
& + &(1 - \lambda)\norm{u_{i+1} - u_{\alpha}^{\delta}}^2
\nonumber\\
& - &\lambda(1-\lambda)\norm{u_{i} - \hu_{i+1}}^2
\nonumber\\
& + &\mu_{i}\left(\mu_{i} - \frac{2}{\norm{T}^2}\right)
\norm{T^{T}T(u_{\alpha}^{\delta} - u_{i})}^2
\nonumber\\
& + & 2\mu_{i} (\alpha - \alpha_{i})
\langle \hu_{i+1} - u_{\alpha}^{\delta} , 
D^{T}w_{\alpha}^{\delta} \rangle
\nonumber\\
& + &2\mu_{i}\alpha_{i}
\langle \hu_{i+1} - u_{\alpha}^{\delta} , 
D^{T}(w_{\alpha}^{\delta} - w_{i}) \rangle .
\nonumber\\
\label{primal_var_total2}
\eea
We now apply Property \ref{property_prox_update} 
for the dual variables 
as in the proof of Proposition \ref{propn_better_approxI},
\bea
\norm{w_{i+1} - w_{\alpha}^{\delta}}^2 \leq \norm{w_{i} - w_{\alpha}^{\delta}}^2 
- \norm{w_{i+1} - w_{i}}^2
& + & 2\nu_{i}\langle w_{i+1} - w_{i} , 
D(\hu_{i+1} - u_{\alpha}^{\delta}) \rangle 
\nonumber\\
& + &2\nu_{i}\langle w_{i} - w_{\alpha}^{\delta} , 
D(\hu_{i+1} - u_{\alpha}^{\delta}) \rangle .
\label{dual_varII_total0}
\eea
We arrive at the following result after summing $\nu_{i}$ times
(\ref{primal_var_total2}) with $\mu_{i}\alpha_{i}$ times
(\ref{dual_varII_total0}),
\bea
\nu_{i}\norm{u_{i + 1} - u_{\alpha}^{\delta}}^2 & \leq & 
\nu_{i}(2-\lambda)\norm{u_{i} - u_{\alpha}^{\delta}}^2 
+ \nu_{i}(\lambda - 1)\norm{\hu_{i+1} - u_{\alpha}^{\delta}}^2
+ \nu_{i}(1 - \lambda)\norm{u_{i+1} - u_{\alpha}^{\delta}}^2
\nonumber\\
& - &\nu_{i}\lambda(1 - \lambda)\norm{ u_{i} - \hu_{i + 1}}^2
+ \mu_{i}\alpha_{i}\norm{w_{i} - w_{\alpha}^{\delta}}^2
- \mu_{i}\alpha_{i}\norm{w_{i+1} - w_{i}}^2
\nonumber\\
& + &\nu_{i}\mu_{i}\left(\mu_{i} - \frac{2}{\norm{T}^2}\right)
\norm{T^{T}T(u_{\alpha}^{\delta} - u_{i})}^2
\nonumber\\
& + &2\nu_{i}\mu_{i}\alpha_{i}\langle w_{i+1} - w_{i} , 
D(\hu_{i+1} - u_{\alpha}^{\delta}) \rangle
\nonumber\\
& + &2\nu_{i}\mu_{i} (\alpha - \alpha_{i})
\langle \hu_{i+1} - u_{\alpha}^{\delta} , 
D^{T}w_{\alpha}^{\delta} \rangle .
\label{best_approx_AlgII}
\eea
As in the previous proof, the inner products are bounded by
\bea
2\nu_{i}\mu_{i}\alpha_{i}\langle w_{i+1} - w_{i} , 
D(\hu_{i+1} - u_{\alpha}^{\delta}) \rangle & \leq & 
(\mu_{i}\alpha_{i})^2\sqrt{\nu_{i}}\norm{w_{i+1} - w_{i}}^2 
+ \sqrt{\nu_{i}}\norm{D(\hu_{i+1} - u_{\alpha}^{\delta})}^2
\nonumber\\
2\nu_{i}\mu_{i} (\alpha - \alpha_{i})
\langle \hu_{i+1} - u_{\alpha}^{\delta} , 
D^{T}w_{\alpha}^{\delta} \rangle & \leq & 
\left(\mu_{i}(\alpha - \alpha_{i})\right)^2\sqrt{\nu_{i}}
\norm{w_{\alpha}^{\delta}}^2 + 
\sqrt{\nu_{i}}\norm{D(\hu_{i+1} - u_{\alpha}^{\delta})}^2
\nonumber
\eea
We now multiply both sides by $\frac{1}{\nu_{i}}$ 
of (\ref{best_approx_AlgII}) with taking into account
the bounds above. So that we obtain,
\bea
\norm{u_{i + 1} - u_{\alpha}^{\delta}}^2 & \leq & 
(2-\lambda)\norm{u_{i} - u_{\alpha}^{\delta}}^2 
+ (\lambda - 1)\norm{\hu_{i+1} - u_{\alpha}^{\delta}}^2
+ (1 - \lambda)\norm{u_{i+1} - u_{\alpha}^{\delta}}^2
\nonumber\\
& - &\lambda(1 - \lambda)\norm{ u_{i} - \hu_{i + 1}}^2
+ \frac{1}{\nu_{i}}\mu_{i}\alpha_{i}\norm{w_{i} - w_{\alpha}^{\delta}}^2
+ \frac{1}{\sqrt{\nu_{i}}}\mu_{i}\alpha_{i}
(\mu_{i}\alpha_{i} - \frac{1}{\sqrt{\nu_{i}}})
\norm{w_{i+1} - w_{i}}^2
\nonumber\\
& + &2\mu_{i}\left(\mu_{i} - \frac{2}{\norm{T}^2}\right)
\norm{T^{T}T(u_{\alpha}^{\delta} - u_{i})}^2
\nonumber\\
& + &\frac{2}{\sqrt{\nu_{i}}}\norm{D(\hu_{i+1} - u_{\alpha}^{\delta})}^2 
+ \frac{1}{\sqrt{\nu_{i}}}\mu_{i}^2(\alpha - \alpha_{i})^2
\norm{w_{\alpha}^{\delta}}^2
\eea
Hence, the result is a natural consequence of the choices of
the asserted parameters.
\end{proof}
\subsection{Convergence of the Iteratively Regularized Approximation Against the Minimum Norm Solution}
\label{iterative_convergence}

This section is dedicated to analyse the convergence 
of the iterative regularized approximation
towards the minimum norm solution. 

%

For each algorithm, we will establish convergence results in the Hadamard
sense on each iterative step $i = 1, 2, \cdots .$ Following up these 
results, cumulative error estimations for both algorithms will be formulated.
In these estimations, exact choices of the step-length $\mu$ 
and the relaxation parameter $\lambda_{i}$ will be conveyed.

\begin{theorem}
\label{trm_algI_conv}
Let the iterative regularization parameter be 
$\alpha_{i}(\delta , v^{\delta}) = \frac{1}{i(\delta , v^{\delta})}$ 
and $\mu_i \leq \frac{2}{\norm{T}^2}.$ Then the
convergence of $u_{i+1}$ produced by Algorithm \ref{algorithmNPA-I}
to the minimum norm solution $u^{\dagger}$ of the linear
inverse problem
$T u^{\dagger} = v^{\dagger}$ which satisfies the VSC
(\ref{variational_ineq}), with the given deterministic noise model 
$\norm{v^{\dagger} - v^{\delta}}\leq\delta$ is satisfied as 
$\delta \rightarrow 0$ and $i \rightarrow \infty.$
\end{theorem}
\begin{proof}
According to Property \ref{property_prox_update},
\bea
\norm{u_{i + 1} - u^{\dagger}}^2 \leq \norm{u_{i} - u^{\dagger}}^2
- \norm{u_{i + 1} - u_{i}}^2 & - & 2\mu_{i} 
\langle u_{i + 1} - u^{\dagger} , T^{T}(T u_{i} - v^{\delta})\rangle
\nonumber\\
& - & 2\mu_{i}\alpha_{i}
\langle u_{i + 1} - u^{\dagger} , D^{T}(w_{i} - w_{0}) \rangle .
\label{conv_AlgI_primal0}
\eea
By means of the deterministic noise model 
$\norm{v^{\dagger} - v^{\delta}}\leq\delta,$ 
consequence of the MDP (\ref{consequence_MDP2}) and the condition
on the step-length
$\mu_i \leq \frac{2}{\norm{T}^2},$
the first inner product on the right hand side is bounded as follows,
\bea
- 2\mu_{i} \langle u_{i + 1} - u^{\dagger} , 
T^{T}(T u_{i} - v^{\delta})\rangle & = & -2\mu_{i}
\langle u_{i + 1} - u_{i} , T^{T}(T u_{i} -v^{\delta}) \rangle
-2\mu_{i}\langle u_{i} - u^{\dagger} , 
T^{T}(T u_{i} -v^{\delta}) \rangle 
\nonumber\\
& = &-2\mu_{i}
\langle u_{i + 1} - u_{i} , T^{T}(T u_{i} -v^{\delta}) \rangle
-2\mu_{i}\langle u_{i} - u^{\dagger} , 
T^{T}(Tu_{i} - Tu^{\dagger}) \rangle 
\nonumber\\
& - & 2\mu_{i} \langle u_{i} - u^{\dagger} , 
T^{T}(v^{\dagger} - v^{\delta}) \rangle
\nonumber\\
& \leq & -2\mu_{i}
\langle u_{i + 1} - u_{i} , T^{T}(T u_{i} -v^{\delta}) \rangle
-2\mu_{i}\norm{Tu_{i} - Tu^{\dagger}}^2
\nonumber\\
& + &2\mu_{i}\delta\norm{T(u_{i} - u^{\dagger})}
\nonumber\\
& \leq & 
-2\mu_{i}\langle u_{i + 1} - u_{i} , T^{T}(T u_{i} -v^{\delta}) \rangle
- 2\mu_{i}\delta^{2}(\tau - 1)^2 
+2\mu_{i}\delta\norm{T (u_{i} - u^{\dagger})},
\nonumber\\
& = &-2\mu_{i}\langle u_{i + 1} - u_{i} , 
T^{T}(T u_{i} - T u^{\dagger}) \rangle 
-2\mu_{i}\langle u_{i + 1} - u_{i} , 
T^{T}(v^{\dagger} - v^{\delta}) \rangle
\nonumber\\
& - & 2\mu_{i}\delta^{2}(\tau - 1)^2 
+2\mu_{i}\delta\norm{T} \norm{u_{i} - u^{\dagger}}
\nonumber\\
& \leq & 2\mu_{i}\norm{u_{i+1} - u_{i}}\norm{T}^2\norm{u_{i} - u^{\dagger}}
+ \norm{u_{i+1} - u_{i}}^2 + \mu_{i}^2\delta^2\norm{T}^2
\nonumber\\
& - & 2\mu_{i}\delta^{2}(\tau - 1)^2 
+2\mu_{i}\delta\norm{T} \norm{u_{i} - u^{\dagger}}
\nonumber\\
& \leq & 4\Psi(\delta)\norm{u_{i+1} - u_{i}} 
+ \norm{u_{i+1} - u_{i}}^2 + 2\delta^2
-\frac{4}{\norm{T}^2}\delta^2(\tau - 1)^2 + 
\delta\frac{4}{\norm{T}}\Psi(\delta).
\nonumber
\eea
Note that we have also used Assumption 
\ref{assump_conventional_variational_ineq} 
for the term $u_{i}$ since by definition $u_{i} \in BV(\Omega).$
Thus, if we plug this estiamation in (\ref{conv_AlgI_primal0}), then
we obtain
\bea
\norm{u_{i + 1} - u^{\dagger}}^2 & \leq & \Psi(\delta)^2
+ 4\Psi(\delta)\norm{u_{i+1} - u_{i}} + 2\delta^2
-\frac{4}{\norm{T}^2}\delta^2(\tau - 1)^2 
\nonumber\\
& + &\delta\frac{4}{\norm{T}}\Psi(\delta)
- 2\mu_{i}\alpha_{i}
\langle u_{i + 1} - u^{\dagger} , D^{T}w_{i}\rangle .
\label{conv_AlgI_primal1}
\eea
Once more, the inner product on the right hand side can suitably 
be bounded by,
\bea
- 2\mu_{i}\alpha_{i}
\langle u_{i + 1} - u^{\dagger} , D^{T}w_{i}\rangle \leq 
\sqrt{\mu_{i}\alpha_{i}}\norm{u_{i + 1} - u^{\dagger}}^2 
+ \sqrt{\mu_{i}\alpha_{i}} \norm{D^{T}w_{i}}.
\nonumber
\eea
Hence, again by the condition on the step-length 
$\mu_i \leq \frac{2}{\norm{T}^2}$ with 
$\alpha_{i}(\delta , v^{\delta}) = \frac{1}{i(\delta , v^{\delta})},$
the following form of (\ref{conv_AlgI_primal1}) yield the result,
\bea
\norm{u_{i + 1} - u^{\dagger}}^2 & \leq & \Psi(\delta)^2
+ 4\Psi(\delta)\norm{u_{i+1} - u_{i}} + 2\delta^2
-\frac{4}{\norm{T}^2}\delta^2(\tau - 1)^2 
\nonumber\\
& + &\delta\frac{4}{\norm{T}}\Psi(\delta)
+ \frac{\sqrt{2}}{\norm{T}}\frac{1}{\sqrt{i}}
\langle u_{i + 1} - u^{\dagger} , D^{T}w_{i}\rangle .
\label{conv_AlgI_primal2}
\eea
\end{proof}
Although in Theorem \ref{trm_algI_conv}, the step-length has been given
in dynamical sense, below we see that in order to obtain desired convergence
we must guarantee that some terms are Ces\'{a}ro summable only when
$\mu_{i}$ is rather chosen fixed.
\begin{theorem}
\label{trm_AlgI_cumulative}
Let the initial guess $u_{0}$ be of $\mathrm{BV}(\Omega)$ and
the dynamical regulariazation parameter that satisfies MDP be
$\alpha_{i}(\delta , v^{\delta}) = \frac{1}{i(\delta , v^{\delta})}.$
If the step-length is chosen as $\mu = \frac{1}{2 i^{\ast}\norm{T}^2},$
then after $i^{\ast}$ times iteration of Algorithm \ref{algorithmNPA-I} 
and in the light of Assumption \ref{assump_conventional_variational_ineq} 
convergence of the iteratively regularized approximation towards the minimum norm 
solution is satisfied, {\em i.e.,}
$\norm{u_{i^{\ast}} - u^{\dagger}} \rightarrow 0$ as 
$i^{\ast} \rightarrow \infty$ whilst $\delta \rightarrow 0.$
\end{theorem}

\begin{proof}
If we iterate the estimation (\ref{conv_AlgI_primal0}) 
from $i = 0$ to $ i = i^{\ast} - 1$ and sum-up over $i,$ 
after simplifications all the necessary terms, we then obtain 
the following cumulative error estimation,
\bea
\norm{u_{i^{\ast}} - u^{\dagger}}^2 \leq \norm{u_{0} - u^{\dagger}}^2
- \sum_{i=0}^{i^{\ast} - 1} \norm{u_{i + 1} - u_{i}}^2 & - & 
2\mu\sum_{i=0}^{i^{\ast} - 1} 
\langle u_{i + 1} - u^{\dagger} , T^{T}(T u_{i} - v^{\delta})\rangle
\nonumber\\
& - & 2\mu\sum_{i=0}^{i^{\ast} - 1}\alpha_{i}
\langle u_{i + 1} - u^{\dagger} , D^{T}w_{i} \rangle .
\label{cumulative_conv_AlgI_primal0}
\eea
Let us begin with rewriting the first inner product on the right hand side,
\bea
- 2\mu\sum_{i=0}^{i^{\ast} - 1}
\langle u_{i + 1} - u^{\dagger} , T^{T}(T u_{i} - v^{\delta})\rangle = 
& - &2\mu\sum_{i=0}^{i^{\ast} - 2} 
\langle u_{i + 1} - u^{\dagger} , T^{T}(T u_{i} - T u^{\dagger})\rangle
\nonumber\\
& - &2\mu\sum_{i=0}^{i^{\ast} - 2} 
\langle u_{i + 1} - u^{\dagger} , T^{T}(v^{\dagger} - v^{\delta})\rangle 
\nonumber\\
& - &2\mu\langle u_{i^{\ast}} - u^{\dagger} , 
T^{T}(T u_{i^{\ast} - 1} - v^{\delta}) \rangle
\nonumber
\eea
With the inclusion of MDP, VSC and the determenistic noise error, 
each piece of the right hand side will be analysed separately after applying
the Cauchy-Schwartz inequality. 
\bea
- 2\mu\sum_{i=0}^{i^{\ast} - 2} 
\langle u_{i + 1} - u^{\dagger} , T^{T}(T u_{i} - T u^{\dagger})\rangle
\leq 2\Psi(\delta)\norm{T}^2\mu\sum_{i=0}^{i^{\ast} - 2}
\norm{u_{i+1} - u^{\dagger}},
\nonumber
\eea
\bea
- 2\mu\sum_{i=0}^{i^{\ast} - 2}
\langle u_{i + 1} - u^{\dagger} , T^{T}(v^{\dagger} - v^{\delta})\rangle
\leq 2\delta\norm{T}\mu\sum_{i=0}^{i^{\ast} - 2}
\norm{u_{i + 1} - u^{\dagger}},
\nonumber
\eea
\bea
- 2\mu\langle u_{i^{\ast}} - u^{\dagger} , 
T^{T}(T u_{i^{\ast} - 1} - v^{\delta}) \rangle \leq 
2\mu\tau\delta\Psi(\delta)\norm{T u_{i^{\ast}-1} - v^{\delta}}.
\nonumber
\eea
We will discuss the convergence of the sequences on the first two lines
as $i^{\ast} \rightarrow \infty$ in the sense of Ces\'{a}ro summation. 
Boundedness of the term $\norm{u_{i+1} - u^{\dagger}}$ has been discussed
already above. Of the convergence conditions formulated in 
Theorem \ref{trm_algI_conv} is the choice of the step-length $\mu.$
In the light of this condition, if we fix the step-length
$\mu = \frac{1}{2i^{\ast}\norm{T}^2},$ then
\bea
- 2\mu\sum_{i=0}^{i^{\ast} - 2} 
\langle u_{i + 1} - u^{\dagger} , T^{T}(T u_{i} - T u^{\dagger})\rangle
\leq \Psi(\delta)\frac{1}{i^{\ast}}\sum_{i=0}^{i^{\ast} - 2}
\norm{u_{i+1} - u^{\dagger}}.
\nonumber
\eea
Thus the sequence on the right hand side is Ces\'{a}ro summable as 
$i^{\ast} \rightarrow \infty.$ Following up this estimation, bounds for 
rest of the terms read
\bea
- 2\mu\sum_{i=0}^{i^{\ast} - 2}\langle u_{i + 1} - u^{\dagger} , 
T^{T}(v^{\dagger} - v^{\delta})\rangle
\leq \frac{\delta}{i^{\ast}\norm{T}} 
\sum_{i=0}^{i^{\ast} - 2}\norm{u_{i + 1} - u^{\dagger}},
\nonumber
\eea
and likewise,
\bea
- 2\mu\langle u_{i^{\ast}} - u^{\dagger} , 
T^{T}(T u_{i^{\ast} - 1} - v^{\delta}) \rangle \leq 
\frac{1}{i^{\ast}\norm{T}}
\Psi(\delta)\norm{T u_{i^{\ast}-1} - v^{\delta}}.
\nonumber
\eea
Also, the last inner product can be bounded again by using
Cauch-Schwartz inequality,
\bea 
- 2\mu\sum_{i=0}^{i^{\ast} - 1}\alpha_{i}
\langle u_{i + 1} - u^{\dagger} , D^{T}w_{i} \rangle \leq 
\frac{1}{(i^{\ast})^{2}\norm{T}^2}\sum_{i=0}^{i^{\ast} - 1}
\norm{u_{i+1} - u^{\dagger}}\norm{D^{T}w_{i}}
\eea
Now, if the negative term drops and the Assumption 
\ref{assump_conventional_variational_ineq} is taken into account,
\bea
\norm{u_{i^{\ast}} - u^{\dagger}}^2 & \leq & \Psi(\delta) + 
\Psi(\delta)\frac{1}{i^{\ast}}\sum_{i=0}^{i^{\ast} - 2}
\norm{u_{i+1} - u^{\dagger}} +  \frac{\delta}{i^{\ast}\norm{T}} 
\sum_{i=0}^{i^{\ast} - 2}\norm{u_{i + 1} - u^{\dagger}}
\nonumber\\
& + & \frac{1}{i^{\ast}\norm{T}}
\Psi(\delta)\norm{T u_{i^{\ast}-1} - v^{\delta}}
 + \frac{1}{(i^{\ast})^{2}\norm{T}^2}\sum_{i=0}^{i^{\ast} - 1}
\norm{u_{i+1} - u^{\dagger}}\norm{D^{T}w_{i}}
\nonumber
\eea
\end{proof}

\begin{theorem}
\label{trm_AlgII_consecutive}
Let the iterative regularization parameter 
$\alpha_{i}(\delta , v^{\delta}) = \frac{1}{i(\delta , v^{\delta})}$
be chosen according to (MDP), 
the relaxation parameter $\lambda \in (1,2)$ of the 
step \ref{projected_linesearch} of the 
Algorithm \ref{algorithmNPA-II} and 
$\mu_i \leq \frac{2}{\norm{T}^2}.$ Then the
convergence of $u_{i+1}$ produced by Algorithm \ref{algorithmNPA-II}
to the minimum norm solution $u^{\dagger}$ of the linear
inverse problem
$T u^{\dagger} = v^{\dagger}$ which satisfies the VSC
(\ref{variational_ineq}), with the given deterministic noise model 
$\norm{v^{\dagger} - v^{\delta}}\leq\delta$ is satisfied as 
$\delta \rightarrow 0$ and $i \rightarrow \infty.$
\end{theorem}
\begin{proof}
According to equality (\ref{strct_cvx_eq}), the error estimation
between the final update $u_{i+1}$ and the minimum norm solution
$u^{\dagger}$ is
\bea
\label{cvx_primal_truthest}
\norm{u_{i + 1} - u^{\dagger}}^2 = 
(1-\lambda)\norm{u_{i} - u^{\dagger}}^2
+ \lambda\norm{\hu_{i+1} - u^{\dagger}}^2
-\lambda(1-\lambda)\norm{u_{i} - \hu_{i+1}}^2
\eea
Also by Property \ref{property_prox_update}, 
some error estimation between
the primal variable and the minimum norm solution
\bea
\norm{\hu_{i+1} - u^{\dagger}}^{2} \leq 
\norm{u_{i} - u^{\dagger}}^{2} - \norm{\hu_{i+1} - u_{i}}
& -2 & \mu_{i}\langle \hu_{i+1} - u^{\dagger} , 
T^{T}(T u_{i} - v^{\delta}) \rangle 
\nonumber\\
& -2 & \mu_{i}\alpha_{i}\langle \hu_{i+1} - u^{\dagger} , 
D^{T}(w_{i} - w_{0}) \rangle .
\label{primal_est_AlgII0}
\eea
Similar to the proof above, the inner products on the right
hand side can easily be handled by replacing $u_{i+1}$ with 
$\hu_{i+1}.$ On account of 
$\mu_i \leq \frac{2}{\norm{T}^2}$ and by definition 
$\hu_{i+1}$ and $u_{i}$ are from the constraint
domain $\mathrm{BV}(\Omega)$, we then obtain,
\bea
- 2\mu_{i} \langle \hu_{i + 1} - u^{\dagger} , 
T^{T}(T u_{i} - v^{\delta})\rangle & = & -2\mu_{i}
\langle \hu_{i + 1} - u_{i} , T^{T}(T u_{i} -v^{\delta}) \rangle
-2\mu_{i}\langle u_{i} - u^{\dagger} , 
T^{T}(T u_{i} -v^{\delta}) \rangle 
\nonumber\\
& = &-2\mu_{i}
\langle \hu_{i + 1} - u_{i} , T^{T}(T u_{i} -v^{\delta}) \rangle
-2\mu_{i}\langle u_{i} - u^{\dagger} , 
T^{T}(Tu_{i} - Tu^{\dagger}) \rangle 
\nonumber\\
& - & 2\mu_{i} \langle u_{i} - u^{\dagger} , 
T^{T}(v^{\dagger} - v^{\delta}) \rangle
\nonumber\\
& \leq & -2\mu_{i}
\langle \hu_{i + 1} - u_{i} , T^{T}(T u_{i} -v^{\delta}) \rangle
-2\mu_{i}\norm{Tu_{i} - Tu^{\dagger}}^2
+ 2\mu_{i}\delta\norm{T(u_{i} - u^{\dagger})}
\nonumber\\
& \leq & 
-2\mu_{i}\langle \hu_{i + 1} - u_{i} , T^{T}(T u_{i} -v^{\delta}) \rangle
- 2\mu_{i}\delta^{2}(\tau - 1)^2 
+2\mu_{i}\delta\norm{T (u_{i} - u^{\dagger})},
\nonumber\\
& = &-2\mu_{i}\langle \hu_{i + 1} - u_{i} , 
T^{T}(T u_{i} - T u^{\dagger}) \rangle 
-2\mu_{i}\langle \hu_{i + 1} - u_{i} , 
T^{T}(v^{\dagger} - v^{\delta}) \rangle
\nonumber\\
& - & 2\mu_{i}\delta^{2}(\tau - 1)^2 
+2\mu_{i}\delta\norm{T} \norm{u_{i} - u^{\dagger}}
\nonumber\\
& \leq & 2\mu_{i}\norm{\hu_{i+1} - u_{i}}\norm{T}^2
\norm{u_{i} - u^{\dagger}} + 
2\mu_{i}\norm{\hu_{i+1} - u_{i}}\norm{T}
\norm{v^{\dagger} - v^{\delta}}
\nonumber\\
& - & 2\mu_{i}\delta^{2}(\tau - 1)^2 
+2\mu_{i}\delta\norm{T} \norm{u_{i} - u^{\dagger}}
\nonumber\\
& \leq &4\Psi(\delta)\norm{\hu_{i+1} - u_{i}} + 
\frac{4\delta}{\norm{T}}\norm{\hu_{i+1} - u_{i}}
- 2\mu_{i}\delta^{2}(\tau - 1)^2 
\nonumber\\
& + &2\mu_{i}\delta\norm{T} \norm{u_{i} - u^{\dagger}}.
\nonumber
\eea
Also, likewise, as for the second inner product on the 
right hand side of (\ref{primal_est_AlgII0}),
\bea
-2 \mu_{i}\alpha_{i}\langle \hu_{i+1} - u^{\dagger} , 
D^{T}(w_{i} - w_{0}) \rangle \leq 
\sqrt{\mu_{i}\alpha_{i}}\norm{\hu_{i+1} - u^{\dagger}}^2
+ \sqrt{\mu_{i}\alpha_{i}}\norm{D^{T}w_{i}}^2.
\nonumber
\eea
Plugging these both estimations into (\ref{primal_est_AlgII0}), 
summing up with (\ref{cvx_primal_truthest}) and 
since the variables $u_{i}$ and $\hu_{i+1}$ are of 
$\mathrm{BV(\Omega)},$ the following estimation yields the result,
\bea
\norm{u_{i+1} - u^{\dagger}}^2 & \leq & (2 - \lambda)
\norm{u_{i} - u^{\dagger}}^2 + (\lambda - 1)
\norm{\hu_{i+1} - u^{\dagger}}^2
+ \left(\lambda(\lambda - 1) - 1\right)\norm{u_{i} - \hu_{i+1}}^2
\nonumber\\
& + & 4\Psi(\delta)\norm{\hu_{i+1} - u_{i}} 
+ 2\delta^2 -\frac{4}{\norm{T}^2}\delta^2(\tau - 1)^2 + 
\delta\frac{4}{\norm{T}}\Psi(\delta)
\nonumber\\
& + &\frac{\sqrt{2}}{\norm{T}}\frac{1}{\sqrt{i}}\Psi(\delta)^2
+ \frac{\sqrt{2}}{\norm{T}}\frac{1}{\sqrt{i}} \norm{D^{T}w_{i}}^2.
\nonumber
\eea
\end{proof}
Now, we are also concerned about the exact choice of the relaxation parameter
$\lambda \in (1,2)$ in Algorithm \ref{algorithmNPA-II}.
As in the proof of Theorem \ref{trm_AlgI_cumulative}, 
Ces\'{a}ro summation of some terms are guaranteed depending 
on the choice of $\lambda$ that is formulated below.
\begin{theorem}
Let the initial guess $u_{0}$ be of $\mathrm{BV}(\Omega)$ and
the dynamical regulariazation parameter that satisfies MDP be
$\alpha_{i}(\delta , v^{\delta}) = \frac{1}{i(\delta , v^{\delta})}.$
If the relaxation parameter $\lambda = \frac{2}{i^{\ast}}$ and step-length
$\mu = \frac{1}{2 i^{\ast}\norm{T}^2},$ then
$\norm{u_{i^{\ast}} - u^{\dagger}} \rightarrow 0$ as 
$i^{\ast} \rightarrow \infty$ whilst $\delta \rightarrow 0.$
\end{theorem}
\begin{proof}
Let us iterate the equality (\ref{cvx_primal_truthest})  
from $i = 0$ to $i = i^{\ast} - 1$ and sum up over $i,$
\bea
\norm{u_{i^{\ast}} - u^{\dagger}}^2 = 
(1 - \lambda)\norm{u_{0} - u^{\dagger}}^2 
& +& (1 - \lambda)\sum_{i = 1}^{i = i^{\ast} - 1}\norm{u_{i} - u^{\dagger}}^2
 + \lambda\norm{\hu_{i^{\ast}} - u^{\dagger}}^2
\nonumber\\
& + &\lambda\sum_{i = 0}^{i = i^{\ast} - 2}\norm{\hu_{i+1} - u^{\dagger}}^2
 - \lambda(1 - \lambda)\sum_{i = 0}^{i = i^{\ast} - 1}\norm{u_{i} - \hu_{i+1}}^2
\nonumber\\
\label{AlgII_cumulative1}
\eea
Note that we have pulled the term2 $\norm{u_{0} - u^{\dagger}}^2$ 
and $\lambda\norm{u_{i^{\ast}} - u^{\dagger}}^2$ out of the sequences 
$(1 - \lambda)\sum_{i = 1}^{i = i^{\ast} - 1}\norm{u_{i} - u^{\dagger}}^2$
and $\lambda\sum_{i = 0}^{i = i^{\ast} - 2}\norm{\hu_{i+1} - u^{\dagger}}^2$
respectively,
so that we can make use of these individuals in the next step.
Then, we repeat the same for the estimation (\ref{primal_est_AlgII0}),
\bea
\norm{\hu_{i^{\ast}} - u^{\dagger}}^2 \leq \norm{u_{0} - u^{\dagger}}^2 
- \sum_{i = 0}^{i = i^{\ast} - 1}\norm{\hu_{i+1} - u_{i}}^2 
& - & 2\mu\sum_{i=0}^{i^{\ast} - 1} 
\langle \hu_{i + 1} - u^{\dagger} , T^{T}(T u_{i} - v^{\delta})\rangle
\nonumber\\
& - & 2\mu\sum_{i=1}^{i^{\ast} - 1}\alpha_{i}
\langle \hu_{i + 1} - u^{\dagger} , D^{T}w_{i} \rangle .
\label{AlgII_cumulative2}
\eea
In this last estimation, analagous sequences of
inner products on the right hand side
have already been bounded in the proof of Theorem \ref{trm_AlgI_cumulative}.
In order to avoid dublicating
the calculations, one can replace $u_{i+1}$ by $\hu_{i+1}.$ Then quick
adaptation of those estimations above must reveal 
\bea
- 2\mu\sum_{i=0}^{i^{\ast} - 1} 
\langle \hu_{i + 1} - u^{\dagger} , T^{T}(T u_{i} - v^{\delta})\rangle
& = & -2\mu\sum_{i=0}^{i^{\ast} - 1} \langle \hu_{i+1} - u^{\dagger} , 
T^{T}T(u_{i} - u^{\dagger})\rangle
- 2\mu\sum_{i=0}^{i^{\ast} - 1} 
\langle \hu_{i + 1} - u^{\dagger} , T^{T}(v^{\dagger} - v^{\delta})\rangle
\nonumber\\
& \leq & 2\mu\norm{T}^2\Psi(\delta)^2\sum_{i=0}^{i^{\ast} - 1}1
+ 2\mu\norm{T}\delta\Psi(\delta)\sum_{i=0}^{i^{\ast} - 1}1
\nonumber\\
& \leq &\Psi(\delta)^2(1 - \frac{1}{i^{\ast}}) + 
\delta\Psi(\delta)(1 - \frac{1}{i^{\ast}})
\label{Alg_II_cumulative_PrimalInner}
\eea
Also likewise,
\bea
- 2\mu\sum_{i=1}^{i^{\ast} - 1}\alpha_{i}
\langle \hu_{i + 1} - u^{\dagger} , D^{T}w_{i} \rangle 
\leq \sqrt{\mu}\sum_{i=1}^{i^{\ast} - 1}\sqrt{\alpha_{i}}
\norm{\hu_{i + 1} - u^{\dagger}}^2
+ \sqrt{\mu}\sum_{i=1}^{i^{\ast} - 1}\sqrt{\alpha_{i}}
\norm{D^{T}w_{i}}^2.
\label{Alg_II_cumulative_DualInner}
\eea
If we sum (\ref{AlgII_cumulative1}) and (\ref{AlgII_cumulative2}) 
with the consideration of both (\ref{Alg_II_cumulative_PrimalInner})
and (\ref{Alg_II_cumulative_DualInner}), after the necessary algebraic 
arrangment, we then arrive at,
\bea
\norm{u_{i^{\ast}} - u^{\dagger}}^2 & \leq & (2 - \lambda)
\norm{u_{0} - u^{\dagger}}^2 + 
(\lambda - 1)\norm{\hu_{i^{\ast}} - u^{\dagger}}^2
+ \lambda\sum_{i = 0}^{i = i^{\ast} - 2}\norm{\hu_{i+1} - u^{\dagger}}^2
\nonumber\\
& - &\left(\lambda(1 - \lambda) + 1\right)
\sum_{i=0}^{i = i^{\ast} - 1}\norm{\hu_{i+1} - u_{i}}^2
\nonumber\\
& + &\Psi(\delta)^2(1 - \frac{1}{i^{\ast}}) + 
\delta\Psi(\delta)(1 - \frac{1}{i^{\ast}}) .
+ \sqrt{\mu}\sum_{i=1}^{i^{\ast} - 1}\sqrt{\alpha_{i}}
\norm{\hu_{i + 1} - u^{\dagger}}^2
+ \sqrt{\mu}\sum_{i=1}^{i^{\ast} - 1}\sqrt{\alpha_{i}}
\norm{D^{T}w_{i}}^2.
\nonumber
\eea
With the given choice of the of the relaxation parameter 
$\lambda$ all the terms on the first and the second lines 
are either bounded or negative. Furthermore, the initial guess 
$u_{0}$ and $\hu_{i^{\ast}}$ are from $\mathrm{BV}(\Omega).$
Hence, all these aforementioned assumptions yield the assertion.
\end{proof}


\section{Numerical Results; Behaviours of the Algorithms From Regularization Aspect}

In this section, we will put the algorithms into tests in order to observe
they behave as iterative regularization prcedures. To this end,
the followings in our numerical tests will be analysed in the computerized 
environment;
\begin{enumerate}
\item Iterative error values on the image space,
\item Iterative error values on the pre-image space,
\item Iterative error values of both algorithms against each other,
\item Iterative error values with the different noise amount input 
on the measured data.
\end{enumerate}
In what follows, the data on the pre-image space are painted drawings, 
see Acknowledgement. The measured data is the sinogram tomographic 
measurement applied on the image, \textbf{\cite{Bardsley18}}.

Firstly, we will provide some simple benchmark on the efficiency
of each algorithm against each other. The mathematical developlment
has proved that Algorithm \ref{algorithmNPA-II} produce 

\begin{figure}[!tbp]
\includegraphics[height=3.5in,width=7in,angle=0]
{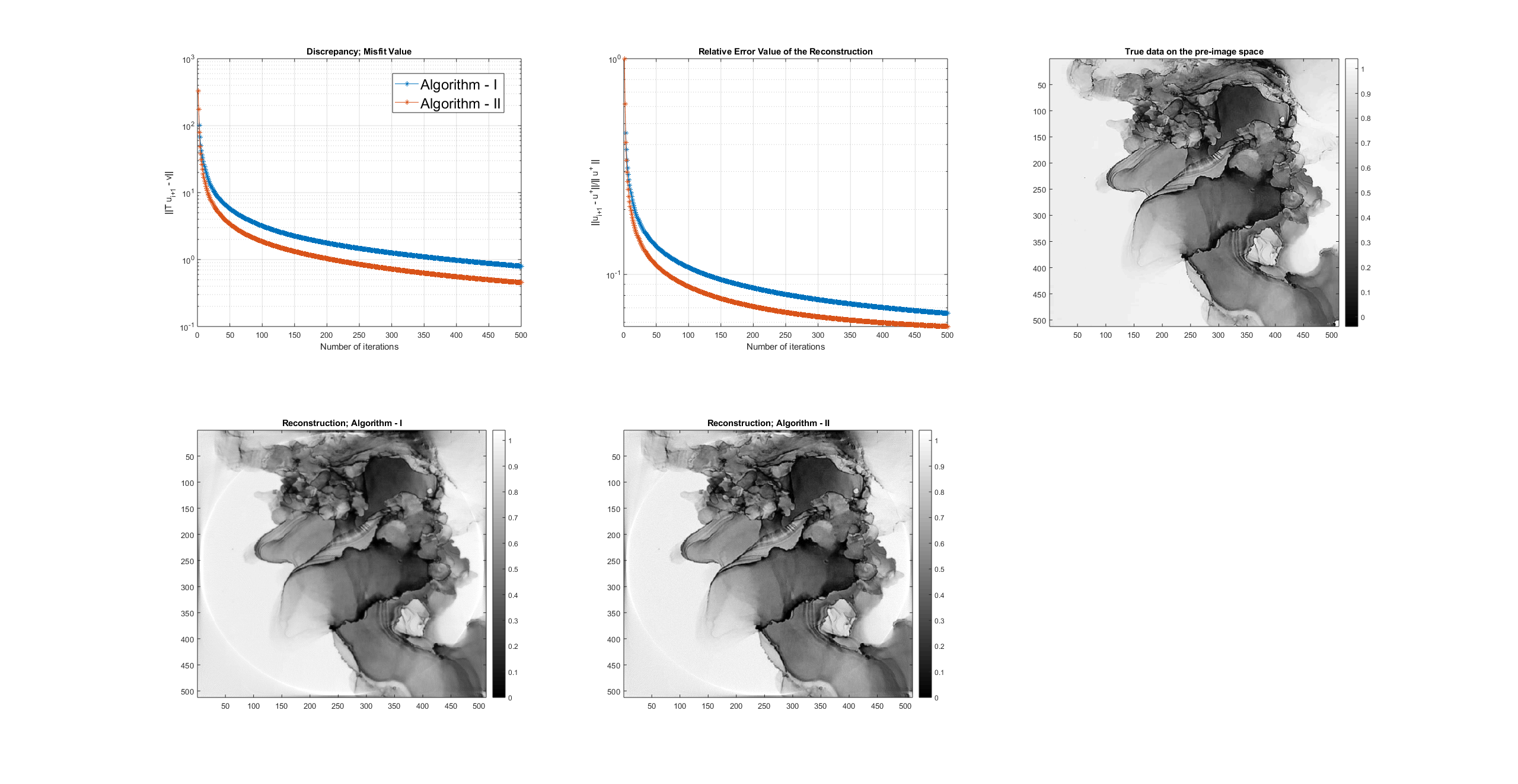}
\caption[Paint reconstruction.]
{\footnotesize Reconstruction of some painting image from its sinogram 
tomography measurement. This benchmark analysis display the efficiency 
of each algorithm. As has been analysed in the mathematical development, 
Algorithm \ref{algorithmNPA-II} provide much less error value.} 
\label{benchmark_AlgIvsII}
\end{figure}
As can be seen in the Figures \ref{benchmark_AlgI_noise} and 
\ref{benchmark_AlgII_noise}, it is verified that the less noise 
amount in the measured data provide less error estimation both 
on the image and the pre-image spaces.
\begin{figure}[!tbp]
\includegraphics[height=4in,width=7in,angle=0]
{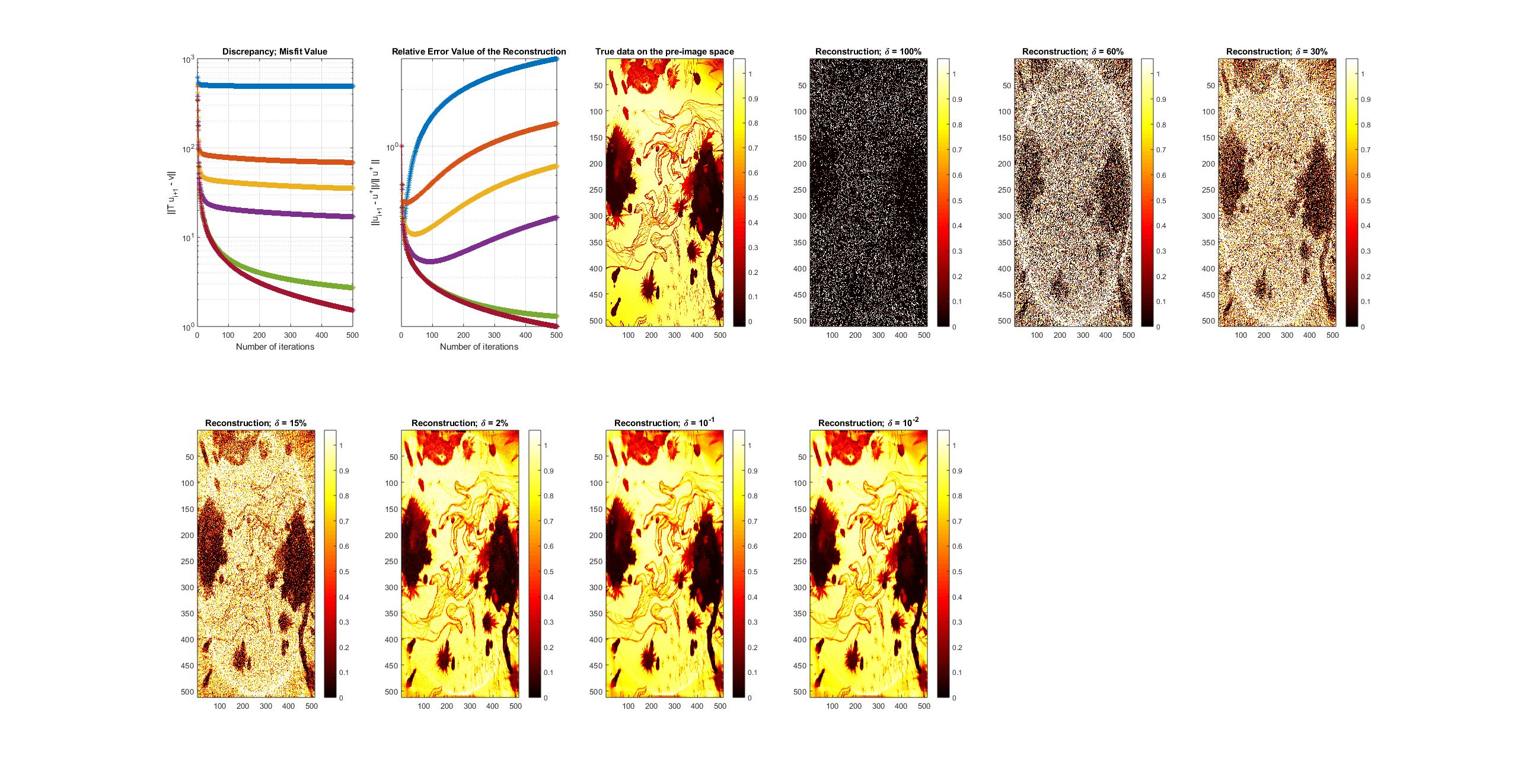}
\caption[Numerical results per different noise amount.]
{\footnotesize Above results display the behaviour of the Algorithm 
\ref{algorithmNPA-I} with the inclusion of different amount of noise
on the measured data. Decay on the error estimations both on the image
and the pre-image spaces has been observed with lesst noise amount.}
\label{benchmark_AlgI_noise}
\end{figure}

\begin{figure}[!tbp]
\includegraphics[height=4in,width=7in,angle=0]
{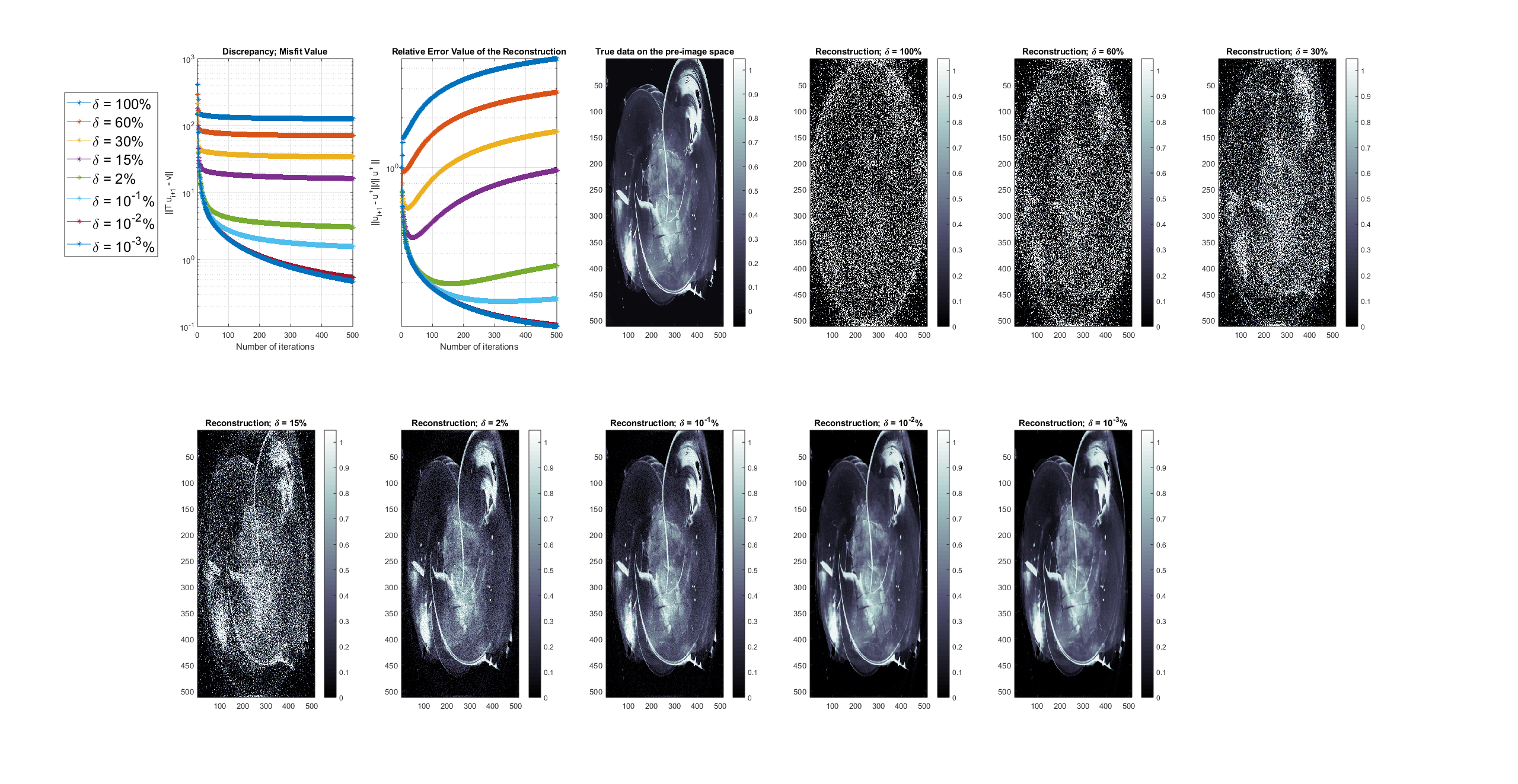}
\caption[Numerical results per different noise amount.]
{\footnotesize Above results display the behaviour of the Algorithm 
\ref{algorithmNPA-II} with the inclusion of different amount of noise
on the measured data. Decay on the error estimations both on the image
and pre-image spaces has been observed with lesst noise amount.}
\label{benchmark_AlgII_noise}
\end{figure}


\section{An Atmospheric Tomography Problem: GPS Tomography}
\label{GPS-tomo}

One important predictor in meteorology is the humidity of the
atmosphere. This is estimated by fan-beam measurements between 
satellite transmitters and land-based receivers. The measurements 
are sparse and fluctuate randomly with receiver availability. 
The task is to reconstruct from these measurements the 3-dimensional, 
spatially varying index of refraction of the atmosphere, 
from which the relative humidity can be inferred.

GPS-tomography involves the reconstruction of some quantity (e.g. humidity), 
pointwise within a volume, from geodesic X-ray measurements transmitted 
by nonuniformly distributed transducers (satellites). These measurements 
are collected by nonuniformly distributed receivers on the ground (ground stations).
As in the conventional tomography, the task here is the reconstruction of the density 
volume profile of a layer in the atmosphere from a set of line integrals
that are of fan-beam projections. 

Function reconstruction from its measured line integrals 
was firstly proposed and solved in \textbf{\cite{Radon17}}. 
Profound mathematical and numerical aspects of the computerized tomography 
have been studied in \textbf{\cite{Natterer01, NattererWuebbeling01}}. 
Measurement from the Radon transform is obtained by integrating
some integrable function over the hyperplanes in $\R^{N}.$ The ray transform,
on the other hand, produces measurement by integrating the function over straight lines.
It is known that in the two dimensional tomography, general Radon and ray transformations 
coincide, \textbf{\cite[p. 17]{NattererWuebbeling01}}.

In the discretized form of the problem, it is assumed that each station receives
equal number of signals transmitted by the satellites. Also for the sake of 
simplicity, we ignore any deviations from the shortest path between transmitters
and receivers due to atmospheric refractivity. The received signal is then
modelled as a line integral along the shortest path between the satellites
and the ground stations.

Peculiar to this problem, reconstructions by Kalman filtering
and ART have been widely applied, 
\textbf{\cite{Bender11, MiidlaRannatUba08, Perler11, ZusBender12}}.
Different from these conventional numerical reconstruction methods, 
a quasi-Newton approach, which is {\em limited memory BFGS} (L-BFGS),
has also been proposed to obtain the optimal regularized solution, 
\textbf{\cite{Altuntac16}}.
The L-BFGS algorithm has been also applied for atmospheric imaging 
wherby the forward problem has been modelled as a phase retrieval problem, see \textbf{\cite{Vogel00}}.

\subsection{Physical Problem}
\label{physical_problem}


This is an inverse problem with incomplete data.
It is well known that the incompleteness of data causes 
nonuniqueness issue in inverse problems,
\textbf{\cite[p. 144]{NattererWuebbeling01}}.
Thorough implementation and inversion of geodesic X-ray
transform has been studied in \textbf{\cite{Monard14}}.
The model of this problem is also widely known. We refer readers
to \textbf{\cite{Altuntac16}} how the compact support
assumption on the targeted data has been established
for satisfying unique solvibility of the problem.

\subsection{Numerical Results; Response of the algorithms to GPS-Tomography}

Above section has already been dedicated to understand that the algorithms
are iterative regularization procedures, furthermore it has been demonstrated
that Algorithm \ref{algorithmNPA-II} provide better results. 
Thus in this section, we will rather present the numerical results 
of Algorithm \ref{algorithmNPA-II} when it is applied to the GPS-Tomography 
problem. Specifically in this problem, recent novel
reserch works \textbf{\cite{Xiong19, Zhao18}} indicate the importance of 
the number and the sources of the measured data. In this work, we are only
able to emphasize that how number of the measurement data, 
{\em e.g.} the signals, affects
the quality of the reconstructed data. In Figure \ref{gps_tomo_FxdParams},
the numerical results have been produced when all the parameters
are chosen as fixed coefficient. Figure \ref{gps_tomo_FxdParams} 
display the results when the parameters are dynamically chosen.
Having a comparison two results against each other, dynamical choice 
of the parameters allow us to obtain better results
with much less iteration steps during the procedure. Regarding the impact
of the number of the measured data on the reconstruction, 
Figures \ref{gps_tomo_MeauseredNo1} and \ref{gps_tomo_MeauseredNo5} 
display less quality on the reconstructed data when the problem
is defined as underdetermined.

In the Figure \ref{gps_tomo_ParamsBenchmark}, we demonstrate the importance of 
the parameter choices. As in the numerical results, with the
dynamically chosen parameters, one can observe less error estimation at the end
of the procedures.

\begin{figure}[!tbp]
\includegraphics[height=3.5in,width=7in,angle=0]
{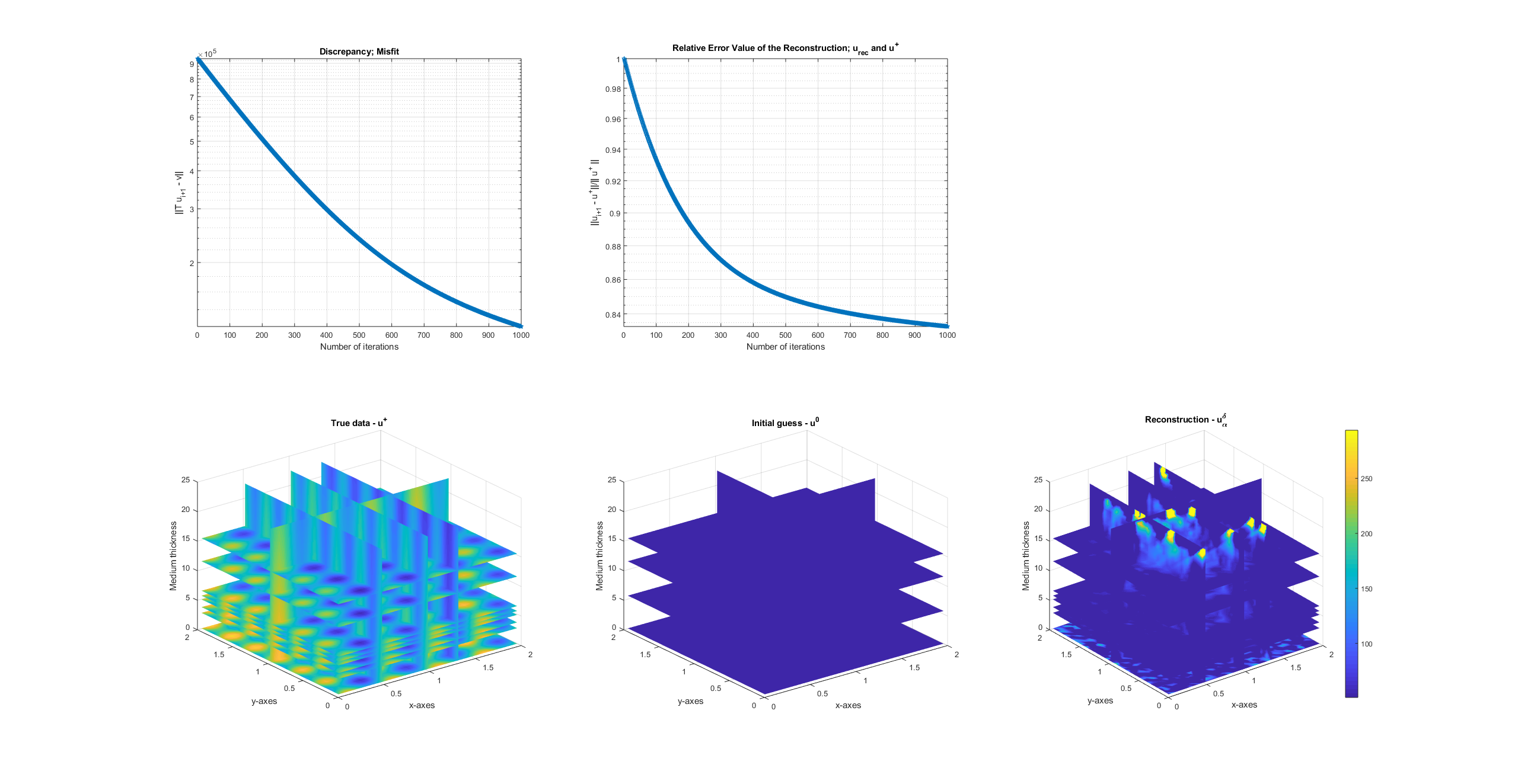}
\caption[Fixed parameters]
{\footnotesize Reconstruction of the volume data when all the parameters 
are chosen as fixed coefficients.}
\label{gps_tomo_FxdParams}
\end{figure}

\begin{figure}[!tbp]
\includegraphics[height=4.5in,width=7in,angle=0]
{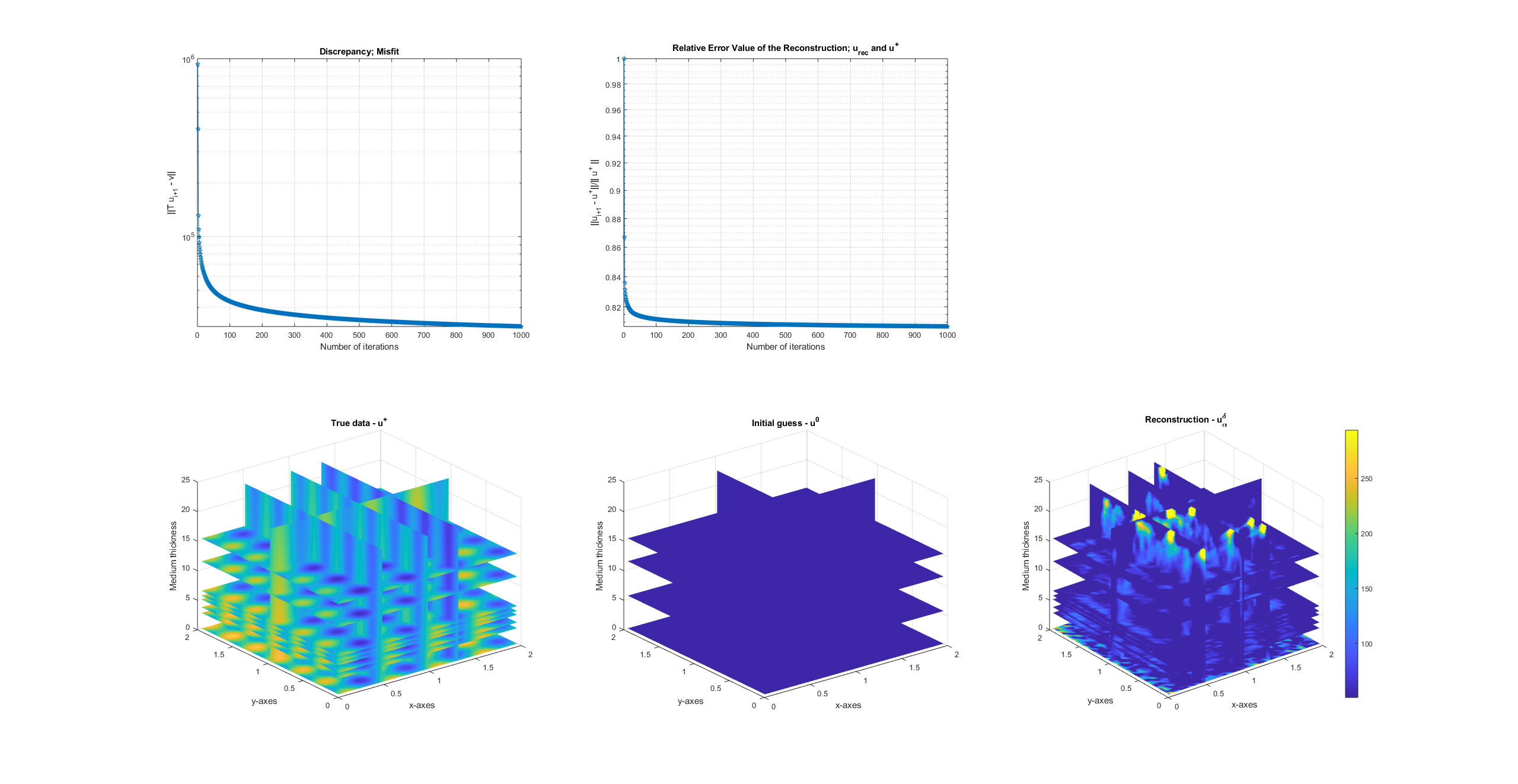}
\caption[Dynamical parameters]
{\footnotesize Reconstruction of the volume data when all the parameters 
are chosen dynamically.}
\label{gps_tomo_DynmParams}
\end{figure}

\begin{figure}[!tbp]
\includegraphics[height=4.5in,width=7in,angle=0]
{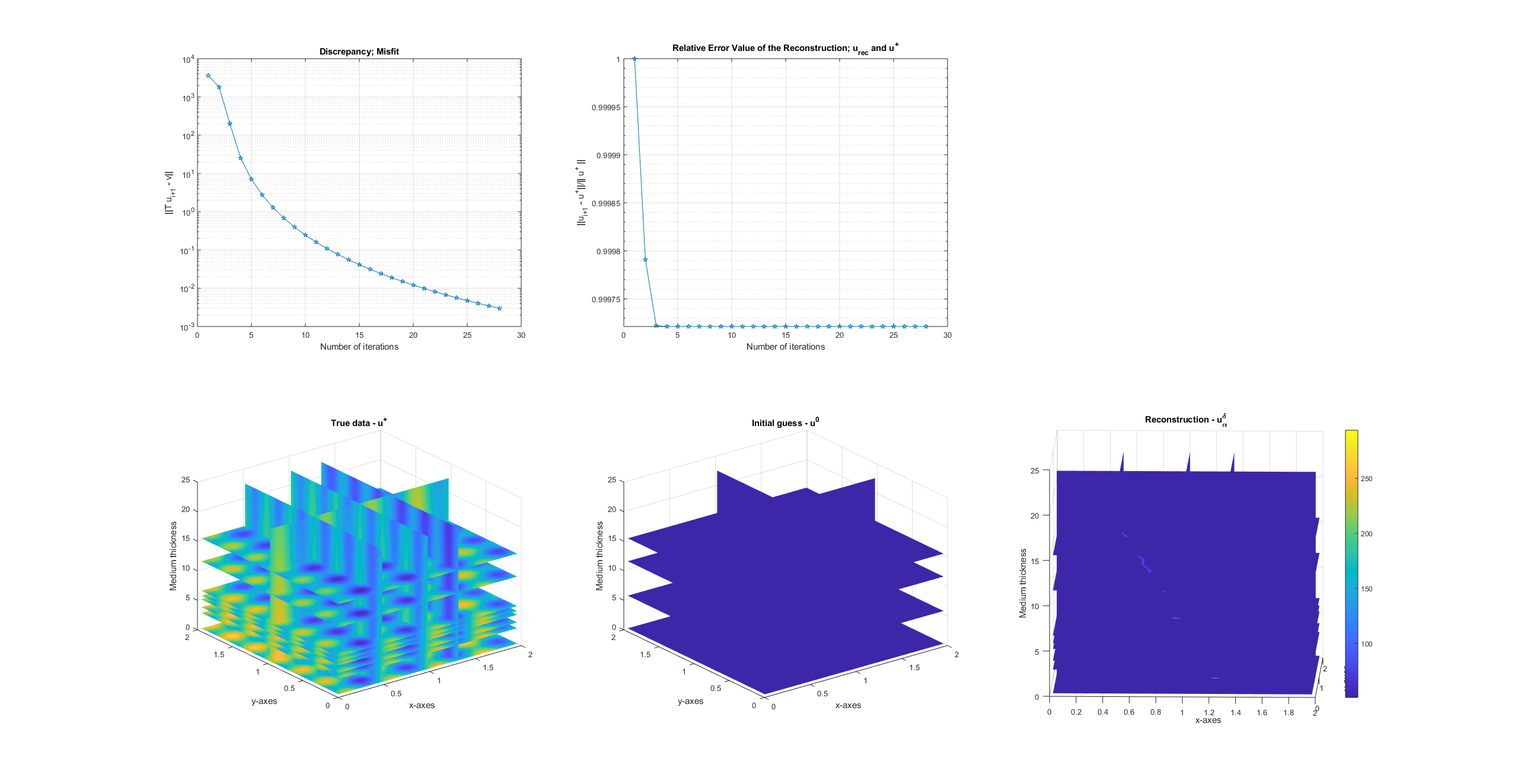}
\caption[Single ray reconstruction]
{\footnotesize Reconstruction of the volume data from single ray. In the 
reconstructed data, the trace of the signal can be seen.}
\label{gps_tomo_MeauseredNo1}
\end{figure}

\begin{figure}[!tbp]
\includegraphics[height=4.5in,width=7in,angle=0]
{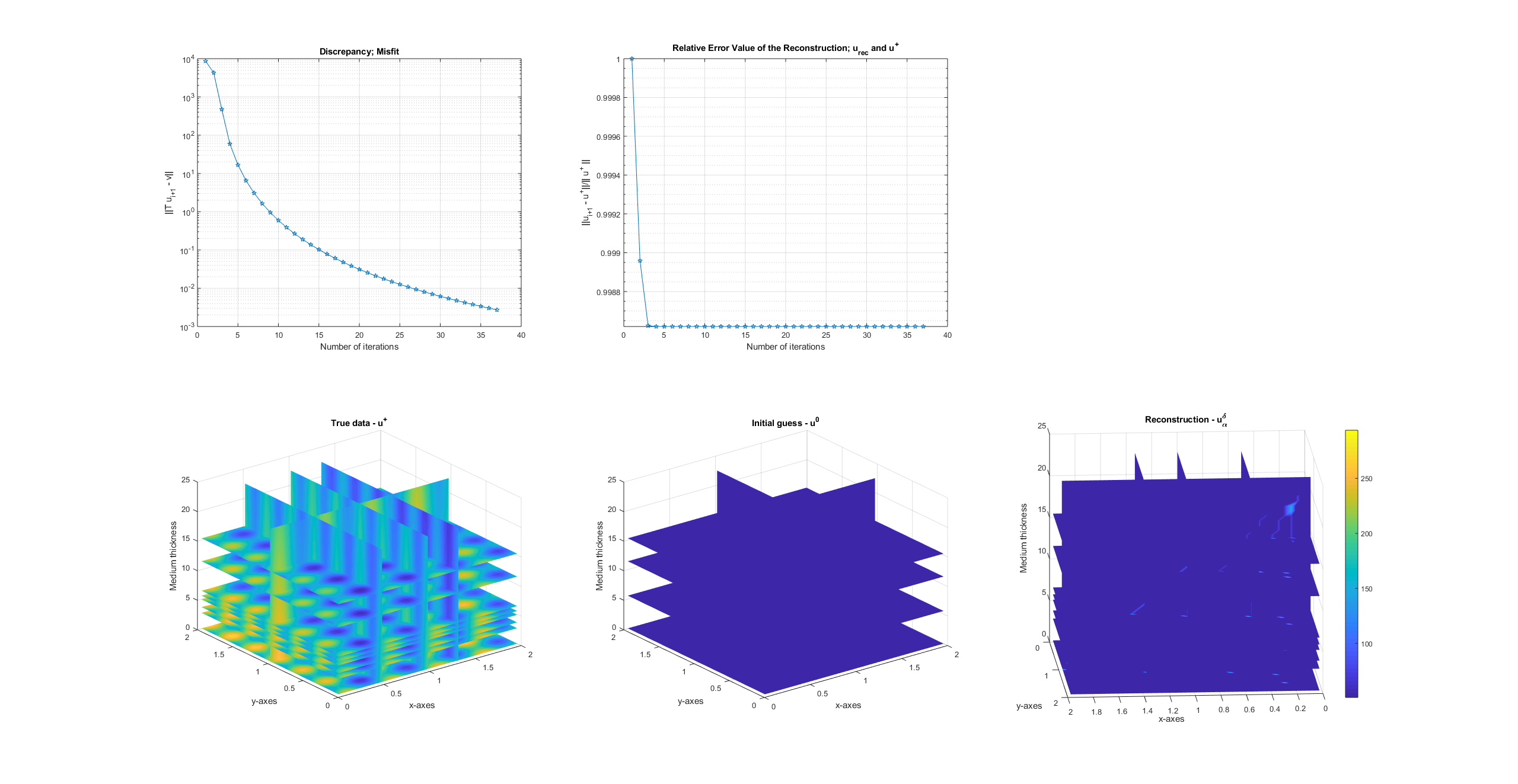}
\caption[Five rays reconstruction]
{\footnotesize Reconstruction of the volume data from five rays.}
\label{gps_tomo_MeauseredNo5}
\end{figure}

\begin{figure}[!tbp]
\includegraphics[height=4.5in,width=7in,angle=0]
{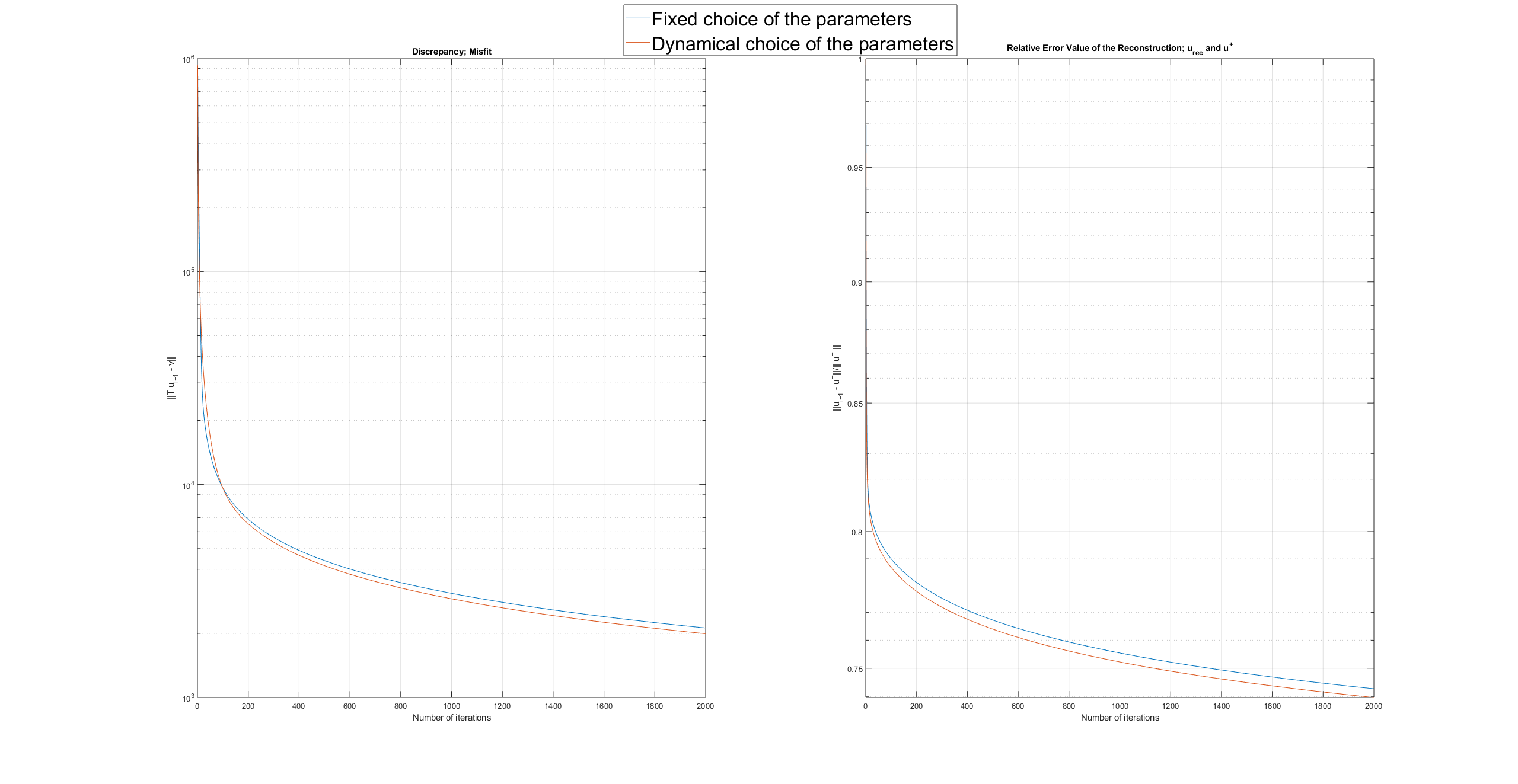}
\caption[Five rays reconstruction]
{\footnotesize Benchmark on dynamical and static parameter choices. Numerical 
results are from Algorithm \ref{algorithmNPA-II} applied to the GPS-Tomography 
problem.}
\label{gps_tomo_ParamsBenchmark}
\end{figure}

\section{Discussion and Future Prospects}
Having the initial guess as constant in the mathematical analysis may not 
reveal the purpose of having Bregman distance as penalizer in our
objective functional (\ref{obj_functional2}). Therefore, 
it could be worthwhile to consider the following iteration scheme
\bea
\nonumber
u_{i+1} \in \argmin_{u \in \cX} \frac{1}{2}\norm{T u - v^{\delta}}^2 +
\alpha_{k} D_{J}(u,u_{k}) + h(u).
\eea
Then, as a result of some quick calculations, the subdifferential
characterization of the iterative minimizer is given by
\bea
\nonumber
\left\{ \begin{array}{rcl}
u_{i+1} &=& \mathrm{prox}_{\mu_{i} h}
\left[u_{i+1}  
- \mu_{i} \left( T^{T}(T u_{i+1} - v^{\delta})
+ \alpha_{i}D^{T}(w_{i+1} - w_{i} ) \right) 
\right]
\\
w_{i+1} &=& \mathrm{prox}_{\nu_{i} g^{\ast}} 
\left( w_{i+1} + \nu_{i} D u_{i+1} \right).
\end{array}\right.
\eea
An algorithm solving this system will definitely convey
minimization impact when the iterative form of Bregman distance
is introduced as penalty term.

\section*{Acknowledgement}
The author is indepted to Ignace Loris for the valuable discussions 
throughout the work. Also, the artist Beng\"{u} U\u{g}urlu 
\"{O}\u{g}retmen is deeply recognized and appreciated for 
sharing the contemporary artwork.

\newpage
\bigskip
\section*{References}


 \bibliographystyle{alpha}

\end{document}